\documentclass[a4paper, 11pt]{article}
\usepackage{graphicx}
\usepackage{color}
\usepackage{mathptmx}      
\usepackage{times}
\usepackage[numbers]{natbib}
\usepackage[T1]{fontenc}
\usepackage[applemac]{inputenc}
\usepackage[all]{xy}
\usepackage{enumerate}
\usepackage{comment}
\usepackage{amsmath,amssymb}
\usepackage[british]{babel}
\usepackage{a4wide}
\usepackage[sans]{dsfont}
\usepackage{amsfonts}
\usepackage{amsthm}
\usepackage[mathscr]{euscript}
\usepackage{mathrsfs}
\usepackage{latexsym}
\usepackage[a4paper, total={6.25in, 8in}]{geometry}

\newcommand{\p}{^}

\newcommand{\X}[1]{X\p{(#1)}}
\newcommand{\XX}[1]{{\tilde{X}}\p{(#1)}}
\newcommand{\N}[1]{X\p{(#1)}}

\newcommand{\pb}[2]{
	\ensuremath{\langle #1,#2 \rangle}}
	\newcommand{\cv}[2]{
	\ensuremath{[#1,#2]}}

\newcommand{\dis}[1]{\displaystyle{#1}}	
\newcommand{\cl}{\textnormal{cl}}
\renewcommand{\d}[1]{{\rm d}#1}
	
\newcommand{\wt}{\widetilde}
\newcommand{\ol}{\overline}

\newcommand{\rmd}{\mathrm{d}}
\newcommand{\rme}{\mathrm{e}}

\newcommand{\m}{\mathrm{m}}

\newcommand{\M}{\mathrm{M}}

\newcommand{\Lrm}{\mathrm{L}}

\newcommand{\Ms}{\overline{\mathrm{M}}}
\newcommand{\Ls}{\overline{L}}
\newcommand{\Ns}{\overline{N}}
\newcommand{\Ws}{\mathrm{W}\p\sig}
\newcommand{\Var}{\mathrm{Var}}

\newcommand{\Bscr}{\mathscr{B}}
\newcommand{\Cscr}{\mathscr{C}}

\newcommand{\Escr}{\mathscr{E}}
\newcommand{\Fscr}{\mathscr{F}}
\newcommand{\Gscr}{\mathscr{G}}
\newcommand{\Hscr}{\mathscr{H}}

\newcommand{\Jscr}{\mathscr{J}}
\newcommand{\Kscr}{\mathscr{K}}
\newcommand{\Lscr}{\mathscr{L}}

\newcommand{\Nscr}{\mathscr{N}}

\newcommand{\Pscr}{\mathscr{P}}

\newcommand{\Rscr}{\mathscr{R}}
\newcommand{\Sscr}{\mathscr{S}}
\newcommand{\Tscr}{\mathscr{T}}

\newcommand{\Vscr}{\mathscr{V}}
\newcommand{\Xscr}{\mathscr{X}}
\newcommand{\Yscr}{\mathscr{Y}}

\newcommand{\Zscr}{\mathscr{Z}}

\newcommand{\zt}{\zeta}
\newcommand{\et}{\eta}

\newcommand{\al}{\alpha}
\newcommand{\bt}{\beta}
\newcommand{\ep}{\varepsilon}
\newcommand{\lm}{\lambda}
\newcommand{\Lm}{\Lambda}
\newcommand{\sig}{\sigma}

\newcommand{\om}{\omega}
\newcommand{\Om}{\Omega}
\newcommand{\cadlag}{c\`adl\`ag }
\newcommand{\Span}{\mathrm{Span}}

\newcommand{\Ebb}{\mathbb{E}}
\newcommand{\EE}{\mathbb{E}}
\newcommand{\Fbb}{\mathbb{F}}
\newcommand{\Nbb}{\mathbb{N}}
\newcommand{\Pbb}{\mathbb{P}}

\newcommand{\Rbb}{\mathbb{R}}
\newcommand{\Zbb}{\mathbb{Z}}

\newcommand{\mc}[1]{\mathcal{#1}}
\newcommand{\mb}[1]{\mathbb{#1}}
\newcommand{\scr}[1]{\mathscr{#1}}

\newcommand{\fraco}[2]{\genfrac{}{}{0pt}{}{#1}{#2}}

\newcommand{\bi}{\begin{itemize}}
\newcommand{\ei}{\end{itemize}}
\newcommand{\bg}{\begin}
\newcommand{\e}{\end}
\newcommand{\be}{\begin{enumerate}}
\newcommand{\ee}{\end{enumerate}}
\newcommand{\beq}{\begin{equation}}
\newcommand{\eeq}{\end{equation}}
\newcommand{\beqs}{\begin{equation*}}
\newcommand{\eeqs}{\end{equation*}}
\newcommand{\beqa}{\begin{eqnarray}}
\newcommand{\eeqa}{\end{eqnarray}}
\newcommand{\beqas}{\begin{eqnarray*}}
\newcommand{\eeqas}{\end{eqnarray*}}

\newcommand{\aP}[1]{
	\ensuremath{\langle #1,#1 \rangle}}
\newcommand{\aPP}[2]{\ensuremath{\langle #1,#2 \rangle}}

\newtheorem{theorem}{Theorem}[section]
\newtheorem{lemma}[theorem]{Lemma}
\newtheorem{proposition}[theorem]{Proposition}
\newtheorem{corollary}[theorem]{Corollary}
\theoremstyle{definition}
\newtheorem{definition}[theorem]{Definition}

\newtheorem{assumption}[theorem]{Assumption}

\newtheorem{remark}[theorem]{Remark}

\begin{document}
\title{The Chaotic Representation Property of Compensated-Covariation Stable Families of Martingales\textsuperscript{\,\scriptsize{a,b}}}
\author{P.\ Di Tella\textsuperscript{\,\scriptsize c},\ \ H.-J.\ Engelbert\textsuperscript{\,\scriptsize d}}
\maketitle
\renewcommand{\thefootnote}{\alph{footnote}}
\footnotetext[1]{\,Work financially supported by the European Community's FP 7 Program under contract PITN-GA-2008-213841, Marie Curie ITN ``Controlled Systems''.}
\footnotetext[2]{\,The first author gratefully acknowledges financial support by Prof.\ Dr.\ Bj\"orn Schmalfu\ss\ of Friedrich-Schiller-Universit\"at Jena.}
\footnotetext[3]{\,E-Mail: Paolo.Di\_Tella@tu-dresden.de. Address: Technische Universit\"at Dresden, Institut f\"ur Stochastik, Zellescher Weg 12--14, 01069 Dresden.}
\footnotetext[4]{\,E-Mail: Hans-Juergen.Engelbert@uni-jena.de. Address: Friedrich-Schiller-Universit\"at Jena, Institut f\"ur Mathematik, Ernst-Abbe-Platz 2 07743.}
\renewcommand{\abstractname}{}
\begin{abstract}

\noindent \small In the present paper, we study the chaotic representation property for certain families $\Xscr$ of square integrable martingales on a finite time interval $[0,T]$. For this purpose, we introduce the notion of compensated-covariation stability of such families. The chaotic representation property will be defined using iterated integrals with respect to a given family $\Xscr$ of square integrable martingales having deterministic mutual predictable covariation $\aPP{X}{Y}$ for all $X,Y\in\Xscr$. The main result of the present paper is stated in Theorem \ref{thm:crp.ccs.fam} below: If $\Xscr$ is a compensated-covariation stable family of square integrable martingales such that $\aPP{X}{Y}$ is deterministic for all $X,Y\in\Xscr$ and, furthermore, the system of monomials generated by $\Xscr$ is total in $L\p2(\Om,\Fscr\p\Xscr_T,\Pbb)$, then $\Xscr$ possesses the chaotic representation property with respect to the $\sigma$-field $\Fscr\p\Xscr_T$. We shall apply this result 
 to the case o
 f L\'evy processes. Relative to the filtration $\Fbb\p L$ generated by a L\'evy process $L$, we construct families of martingales which possess the chaotic representation property. As an illustration of the general results, we will also discuss applications to continuous Gaussian families of martingales and independent families of compensated Poisson processes. We conclude the paper by giving, for the case of L\'evy processes, several examples of concrete families $\Xscr$ of martingales including Teugels martingales.
\end{abstract}

%
%

\def\thefootnote{\arabic{footnote}}

\section{Introduction}\label{sec:intro}
In his paper \cite{Wi38}, Norbert Wiener introduced the notion of \emph{multiple integral} and called it \emph{polynomial chaos}. However, the Wiener polynomials chaos of different order are not orthogonal. In \cite{I51}, It\^o gave another definition of multiple integrals for a general \emph{normal random measure} in such a way that  the orthogonality property is achieved. In the same paper, It\^o established the relation between orthogonal Hermite polynomials and  multiple integrals. This was also done by \citet{CM47} for the special normal random measure induced by a Wiener process. Using this relation, It\^o proved that every square integrable functional of a normal random measure can be expanded as an orthogonal sum of multiple integrals. This property is known in the literature as \emph{chaotic representation property} (CRP). A similar result was shown in \citet{Ka50}. In conclusion of \cite{I51}, It\^o pointed out that the multiple integrals of a normal random measure induced 
 by a Wiener process $W$ can be regarded as \emph{iterated stochastic integrals} with respect to $W$. In the later paper \cite{I56}, which appeared in 1956, It\^o generalized the result of \cite{I51} stated for a normal random measure, and in particular for the Wiener process, to the case of an \emph{orthogonal} random measure (cf.\ \citet{GS74}, Chapter IV, \S\ 4) defined as a sum of a normal random measure and a \emph{compensated Poisson random measure}. These random measures are associated with processes with independent increments. For such a random measure, \citet{I56} introduced multiple integrals and proved a chaos decomposition for square integrable functionals. We call the multiple integrals introduced in \cite{I56} multiple It\^o integrals. In \cite{I56} no relation between multiple It\^o integrals and iterated stochastic integrals is given. However, we note that the relation between multiple It\^o integrals and polynomials is established by \citet{SK76}.

Multiple It\^o integrals have been extensively studied for many decades. A self-contained monograph about It\^o-type integrals for \emph{completely random measures}, with particular attention to their combinatorial structure, is \citet{PT11}. In \cite{PT11} the authors also discuss the CRP for centred Gaussian measures and compensated Poisson random measures (which are special cases of completely random measures) and study the relation of multiple integrals with Hermite polynomials for the Gaussian case and Charlier polynomials for the Poisson case.

For a stochastic process $X$ we denote by $\Fbb\p X$ the smallest filtration satisfying the usual conditions such that $X$ is adapted.

Let now $X$ be a square integrable martingale on the finite horizon $[0,T]$, $T>0$. To approach the problem if the martingale $X$ possesses the CRP with respect to the space $L\p2(\Om,\Fscr\p X_T,\Pbb)$, it is necessary to define \textit{multiple integrals} with respect to $X$. In general, this is not an easy task because it is not clear how to associate a suitable random measure with $X$ allowing to introduce multiple integrals. For this reason it is more convenient to introduce \textit{iterated stochastic integrals} with respect to $X$. But this, as observed in \citet{M76}, p.\ 321--331, is in general not straightforward.  However, if $X$ is a square integrable martingale such that $\aPP{X}{X}_t=t$, $t\geq0$, and $F$ is a square integrable deterministic function on $[0,T]\p n$, then the $n$-fold iterated integral $J_n(F)_T$ of $F$ up to time $T$ is well defined (cf.\ \citet{M76}, p.\ 325--327). Note that for $n\neq m$, $J_n(F)_T$ and $J_m(G)_T$ are orthogonal. We recall that the so
  called Az\'ema--Yor martingales, which were introduced in \citet{A85} and \citet{AY89}, are of this type. Using the \emph{structure equation} as a tool, Emery \cite{E89} has shown that some of the Az\'ema--Yor martingales possess the CRP, meaning that $L\p2(\Om,\Fscr_T\p X,\Pbb)$ allows a decomposition into the orthogonal sum of the linear subspaces of $n$-fold iterated integrals associated with $X$.

In this paper we deal with the CRP of certain families $\Xscr$ of square integrable martingales instead of only single processes $X$. We shall restrict ourselves to a finite time horizon $[0,T]$, $T>0$, and to the filtration $\Fbb\p\Xscr$, that is, the smallest filtration satisfying the usual conditions with respect to which $\Xscr$ is a family of adapted processes.
The study of the CRP for families $\Xscr$ of square integrable martingales turns out to be of major interest. This is because the CRP for a single square integrable martingale is a strong property and a relatively small class of processes possesses it.

One important example of a family of square integrable martingales possessing the CRP has been considered by \citet{NS00}. Let $L$ be a L\'evy process with L\'evy measure $\nu$. Under the assumption that $\nu$ has a finite exponential moment outside the origin, the authors define the family of orthogonalized Teugels martingales, which is a family consisting of countably many orthogonal square integrable $\Fbb\p L$-martingales. Then the system of iterated integrals generated by the orthogonalized Teugels martingales is introduced and it is shown that this family of martingales possesses the CRP on $L\p2(\Om,\Fscr\p L_T,\Pbb)$ (for the precise definition of the CRP see Definition \ref{def:crp} below). Notice that in \citet{NS00} the assumption on the L\'evy measure is rather strong. Furthermore in \citet{NS00} the relationship between the iterated integrals generated by the orthogonalized Teugels martingales and the multiple It\^o integrals introduced in \cite{I56} is not studied. The 
 problem was also mentioned in \citet{SUV05}, where in Proposition 7 the relationship between the iterated integrals generated by the orthogonalized Teugels martingales and the multiple It\^o integrals is stated without proof.

The aim of this paper is to study the CRP for certain families $\Xscr$ of square integrable martingales such that the process $\aPP{X}{Y}$ is deterministic whenever $X$ and $Y$ belong to $\Xscr$. Note that the martingales in $\Xscr$ need not have independent increments: If a square integrable martingale $X$ has independent increments, then $\aPP{X}{X}$ is deterministic, the converse is however not true. Most of normal martingales $X$ (i.e., square integrable martingales $X$ such that $\aP{X}_t=t$, $t\geq 0$) and, in particular, solutions of the structure equation do not have independent increments (cf.\ \citet{E89} where the very special case of independent increments is discussed on p.\ 74). For the family $\Xscr$ as above the iterated integrals can be defined, and we shall look for sufficient conditions to ensure that $\Xscr$ possesses the CRP as it will be introduced in Definition \ref{def:crp} below. We shall require that the family $\Xscr$ is \emph{compensated-covariation stable
 }, i.e., that for every $X,Y\in\Xscr$ the process $[X,Y]-\aPP{X}{Y}$ again belongs to $\Xscr$, $[X,Y]$ denoting the covariation process of $X$ and $Y$. This property of compensated-covariation stability of families of martingales has been introduced in \citet{DT13} and \citet{DE13}, where the \emph{predictable representation property} is studied.

In Section \ref{sec:prel} we recall some basic definitions and notations from stochastic analysis needed in the following sections of the paper.

Given a family $\Xscr$ of square integrable martingales with deterministic predictable covariation $\aPP{X}{Y}$ for all $X$ and $Y$ from $\Xscr$, in Section \ref{sec:def.it.int} the iterated integrals are introduced and their properties are studied. Furthermore the definition of the CRP for such kind of families is given.

In Section \ref{sec:ccs.f.mul.int} some more important properties of the iterated stochastic integrals are obtained under the further assumption that $\Xscr$ is a compensated-covariation stable family.

The main result of this paper is proven in Section \ref{sec:chaos} (see Theorem \ref{thm:crp.ccs.fam} below): If $\Xscr$ is a compen\-sated-covariation stable family of square integrable martingales such that $\aPP{X}{Y}$ is deterministic for all $X,Y\in\Xscr$, and furthermore the system of monomials generated by $\Xscr$ is total in $L\p2(\Om,\Fscr\p\Xscr_T,\Pbb)$, then $\Xscr$ possesses the CRP.

Section \ref{sec:prp.lev} is devoted to applications of the general results established in the previous sections to L\'evy processes. Let $L$ be a L\'evy process with L\'evy measure $\nu$ and Gaussian part $\Ws$, where $\Ebb[(\Ws_t)\p2]=\sig\p2 t$, $\sig\p2\geq0$, and $\mu:=\sig\p2\delta_0+\nu$, $\delta_0$ being the Dirac measure in zero. With a deterministic function $f$ in $L\p2(\mu)$ we associate a square integrable martingale $\X{f}$ by setting $\X{f}_t:=f(0)\Ws_t+\Ms(1_{[0,t]}f)$, where $\Ms(1_{[0,t]}f)$ denotes the stochastic integral of $1_{[0,t]}f$ with respect to the compensated Poisson random measure $\Ms$ associated with the jumps of $L$. We prove that for a system $\Tscr$ in $L\p2(\mu)$ the associated family $\Xscr_\Tscr:=\{\X{f},\ f\in\Tscr\}$ possesses the CRP on $L\p2(\Om,\Fscr\p L_T,\Pbb)$ \emph{if and only if} $\Tscr$ is total (i.e., its linear hull is dense) in $L\p2(\mu)$ (cf. Theorem \ref{thm:tot.sys.h2.prp} below). A particularly important situation occurs when $
 \Tscr$ is a complete orthogonal system: In this case $\Xscr_\Tscr$ is a family of orthogonal martingales and we shall see that this simplifies the CRP considerably. This is a major generalisation of \citet{NS00} because we are able to construct a great variety of families of martingales possessing the CRP for \textit{any} L\'evy process, without any assumption on the L\'evy measure, therefore also in the case if Teugels martingales cannot be introduced.
Then, for a total system $\Tscr$ in $L\p2(\mu)$, we investigate the relationship between the iterated integrals generated by $\Xscr$ and the multiple It\^o  integrals as well as between the CRP and the chaos expansion obtained in \citet{I56}. 

Finally, as an illustration of the general results in Section \ref{sec:appl} several applications will be given. We start with Gaussian families of continuous local martingales and pass on to independent families of Poisson processes. Then we proceed with examples for concrete families of martingales constructed from L\'evy processes, including the family of Teugels martingales as a particular case.

\section{Basic Definitions and Notations}\label{sec:prel}
In this section we recall some basic definitions and notations from stochastic analysis needed in the following sections of the paper. By $(\Om,\Fscr,\Pbb)$ we denote a \emph{complete} probability space and by $\Fbb$ a filtration satisfying the \emph{usual conditions}. We shall always consider real-valued stochastic processes on a finite time horizon $[0,T]$, $T>0$. 

With a c\`adl\`ag process $X=(X_{\,t})_{t\in[0,T]}$, we associate the process $X_{-}=(X_{\,t-})_{t\in[0,T]}$ setting $X_{\,0-}=0$ and $X_{\,t-}=\lim_{s\uparrow t}X_{\,s}$, $t>0$. The process $\Delta X=(\Delta X_{\,t})_{t\in[0,T]}$, $\Delta X_{\,t}:=X_{\,t}-X_{\,t-}$, $t\in[0,T]$, is called  \emph{the jump process} of $X$. Because of the definition of $X_{0-}$, we always have $\Delta X_{\,0}=X_0$.

In the present paper, $\Fbb$-martingales are always assumed to be c\`adl\`ag. For a martingale $X$, by $X\p c$ and $X\p d$ we denote the continuous and the purely discontinuous martingale part of $X$, respectively (cf.\ \citet{JS00}, Theorem I.4.18). We recall that $X\p c_0=X\p d_0=0$.

We say that a martingale $X$ is \emph{square integrable} if $X_{\,T}\p 2$ is integrable. Because of Doob's inequality, this is equivalent to require that $\sup_{t\in[0,T]}X_{\,t}\p2$ is integrable. By $\Hscr\p{\,2}=\Hscr\p{\,2}(\Fbb)$ we denote the set of square integrable martingales and by $\Hscr_0\p2$ the subspace of the elements of $\Hscr\p 2$ starting at $0$. We set $\|X\|_{\Hscr\p{\,2}}:=\|X_T\|_2$, where $X\in\Hscr\p{\,2}$ and $\|\cdot\|_2$ denotes the $L\p{\,2}(\Pbb):=L\p{\,2}(\Om,\Fscr,\Pbb)$-norm. For $X,Y\in\Hscr\p{\,2}$ we put $(X,Y)_{\Hscr\p{\,2}}:=\Ebb[X_T Y_T]$ which defines a scalar product on $\Hscr\p{\,2}$.
We can identify $(\Hscr\p{\,2},\|\cdot\|_{\Hscr\p{\,2}})$ with the space $(L\p{\,2}(\Pbb),\|\cdot\|_2)$.

If $X$ and $Y$ belong to $\Hscr\p{\,2}$, we say that they are \emph{orthogonal} and write $X\bot Y$ if their product $XY$ is a martingale with $X_0Y_0=0$ (cf.\ \cite{J79}, Definition 2.10).
If $\Xscr\subseteq\Hscr\p2$, we say that $Y\in\Hscr\p2$ is orthogonal to $\Xscr$ if $Y\bot\,X$ for every $X\in\Xscr$. We stress that if $X,Y\in\Hscr\p{\,2}$ are orthogonal, then $X_{\,t}$ and $Y_{\,t}$ are orthogonal in $L\p2(\Pbb)$, for every $t\in[0,T]$. However, the converse is, in general, not true (cf.\ \citet{P05}, p.\ 181).

By $\Vscr$ we denote the set of adapted \cadlag processes with paths of finite variation on $[0,T]$. If $A\in\Vscr$, then we say that $A$ is a \emph{process of finite variation}. If $A\in\Vscr$, by $\Var(A)=(\Var(A)_{t})_{t\in[0,T]}$ we denote the variation process of $A$. We say that $A$ is of integrable variation if $\Var(A)_T$ is integrable.

For $A\in\Vscr$ we shall make use of the Riemann--Stieltjes integral of a measurable process $H$ with respect to $A$ (cf.\ \cite{JS00}, I, \S\ 3a):
If $\int_0\p t|H_{\,s}(\om)|\,\rmd \Var(A)_{\,s}(\om)<+\infty$, we use the notation $H\cdot A_{\,t}(\om):=\int_0\p tH_{\,s}(\om)\,\rmd A_{\,s}(\om)$ and otherwise $H\cdot A_{\,t}(\om):=+\infty$. We write $H\cdot A=(H\cdot A_{\,t})_{t\in[0,T]}$ for the integral process.
We stress that if $A\in\Vscr$ and $H$ is a measurable process, then $H\cdot A$ belongs to $\Vscr$ if and only if $H\cdot A_{\,t}(\om)$ is finite-valued, i.e., $\int_0\p t|H_{\,s}(\om)|\,\rmd \Var(A)_{\,s}(\om)<+\infty$, for every $t\in[0,T]$ and $\om\in\Om$.

If $X$ and $Y$ belong to $\Hscr\p{\,2}$ there exists a unique predictable process of integrable variation, denoted by $\pb{X}{Y}$ and called the \emph{predictable covariation} of $X$ and $Y$, such that $\pb{X}{Y}_0=0$ and $XY-\pb{X}{Y}$ is a martingale (cf.\ \citet{JS00}, Theorem I.4.2). Clearly $\Ebb[X_T Y_T-X_0Y_0]=\Ebb[\pb{X}{Y}_{\,T}]$ and $X,Y\in\Hscr\p{\,2}$ are orthogonal if and only if $\pb{X}{Y}=0$ and $X_0Y_0=0$. 

Let $(X,\Fbb)$ be a semimartingale with decomposition $X=X_{\,0}+M+A$, where $M$ (without loss of generality) is \emph{locally} in $\Hscr\p{\,2}_0$, $A\in\Vscr$ with $A_0=0$, and $X_{\,0}$ is $\Fscr_0$-measurable. The continuous martingale part of $X$, denoted by $X\p c$, is defined by $X\p c:=M\p c$. Note that $X\p c$ does not depend on the semimartingale decomposition (cf.\ \citet{JS00}, Proposition I.4.27). With two semimartingales $X$ and $Y$, we associate the process $[X,Y]$, called covariation of $X$ and $Y$, defining
\begin{equation}\label{def:cov}
[X,Y]_{\,t}:=\pb{X\p c}{Y\p c}_{\,t}+\sum_{0\leq s\leq t}\Delta X_{\,s}\Delta Y_{\,s},\quad t\in[0,T].
\end{equation}
It is well-known that the process $[X,Y]$ belongs to $\Vscr$ (cf.\ \citet{J79}, Theorem 2.30). We remark that the definition \eqref{def:cov} of the covariation $[X,Y]$ implies that $[X,Y]_0=X_0Y_0$.
If $X,Y\in\Hscr\p{\,2}$, then $[X,Y]$ is of integrable variation and $\pb{X}{Y}$ is the compensator of $[X,Y]$, i.e., $\pb{X}{Y}$ is the unique predictable process of integrable variation starting at zero such that $[X,Y]-\pb{X}{Y}$ is a martingale.

Now we are going to recall the stochastic integral with respect to a martingale $X\in\Hscr\p{\,2}$. The space of integrands for $X$ is given by
$
\Lrm\p{2}(X):=\{H\textnormal{ predictable}: \Ebb[H\p{\,2}\cdot\pb XX_T]<+\infty\}
$.
For $X\in\Hscr\p{\,2}$ and $H\in\Lrm\p{2}(X)$, by $H\cdot X$ we denote the stochastic integral of $H$ with respect to $X$. The stochastic integral of $H$ with respect to $X$ is characterized as it follows: Let $Z\in\Hscr\p2$. Then $Z=H\cdot X$ if and only if $Z_0=H_0X_0$ and $\pb ZY=H\cdot\pb XY$, for every $Y\in\Hscr\p{\,2}$. We stress that, if $X,Y\in\Hscr\p{\,2}$ are orthogonal martingales, then also $H\cdot X$ and $K\cdot Y$ are orthogonal martingales of $\Hscr\p{\,2}$, for all $H\in\Lrm\p{2}(X)$ and $K\in\Lrm\p{2}(Y)$. The notation $H\cdot X$ is not ambiguous with the one introduced for the stochastic integral with respect to a process of finite variation: If $X\in\Hscr\p{\,2}\cap\Vscr$ then $H\cdot X$ coincides with the Riemann--Stieltjes integral (cf.\ \citet{J79}, Remark 2.47). 
 
For a set $\Kscr$ of a Banach space $(\Hscr,\|\cdot\|)$, by $\Span(\Kscr)$ we denote the linear hull of $\Kscr$ and by  $\cl(\Kscr)_{\Hscr}$ the closure of $\Kscr$ in $\Hscr$.

\section{Iterated Integrals and Chaotic Representation Property}\label{sec:def.it.int}
Let $\Xscr\subseteq\Hscr\p{\,2}(\Fbb)$ be a family of $\Fbb$-martingales. For notational convenience, we represent $\Xscr$ in parametric form: $\Xscr:=\{\N{\al},\ \al\in \Lm\}$, where $\Lm$ is an associated index set. In this paper, from now on we shall always assume that $\aPP{\X{\al}}{\X{\bt}}$ is deterministic, $\al,\bt\in\Lm$. For such a family we are going to introduce the \emph{iterated integrals}.

Let $F_0$ be a bounded $\mc F_0$-measurable function and $F_1,\ldots, F_n$  bounded measurable functions on the measurable space $([0,T],\Bscr([0,T]))$. We denote by $F:=F_0\otimes\cdots\otimes F_n=\otimes_{k=0}^n F_k$ the \emph{tensor product} of $F_0,\ldots,F_n$ defined on $\Om\times[0,T]^n$ and say that $F$ is an \emph{elementary function} of order $n$.

\begin{definition}\label{elem.it.int.}
Let $\al_1, \ldots, \al_n\in \Lm$ be given and $F=F_0\otimes\cdots\otimes F_n$ be an elementary function of order $n$. 
The \emph{elementary} iterated integral $J_0(F_0)$ of order zero of $F_0$ and $J_n\p{(\al_1,\ldots, \al_{n})}(F)$ of order $n$ of $F$ ($n\geq 1$) with respect to the martingales ($\X{\al_1}, \ldots\ , \X{\al_n}$) is defined inductively as follows: If $n=0$, then the process $J_0(F_0):=(J_0(F_0)_{\,t})_{t\in [0,T]}$ is defined by $J_0(F_0)_t=F_0$ for $t\in[0,T]$ and, for all $1\leq m\leq n$,
\begin{equation}\label{eq:def.sim.mul.int}
J_{m}\p{(\al_1,\ldots, \al_{m})}(F_0\otimes\cdots\otimes F_{m})_{\,t}
:=\int_0^t J_{m-1}^{(\al_1,\ldots, \al_{m-1})}(F_0\otimes\cdots\otimes F_{m-1})_{u-}\,F_{m}(u)\, \rmd {\X{\al_{m}}_u},\ \ t\in [0,T]\,.
\end{equation}
\end{definition}
In the following lemma we establish some important properties of the iterated integrals.
\begin{lemma}\label{lem:isom.sim.mul.int} Let $n\geq0$, $\al_1,\ldots,\al_n\in\Lm$ and $F=\otimes_{k=0}^n F_k$ be an elementary function of order $n$.

\textnormal{(i)} The elementary iterated integral $J_{n}\p{(\al_1,\ldots,\al_n)}(F)$ belongs to $\Hscr\p{\,2}$ for $n=0$ and to $\Hscr\p{\,2}_0$ for $n\geq1$.

\textnormal{(ii)} Let moreover be $m\geq0$, $\bt_1, \ldots, \bt_m\in \Lm$ and $G=\otimes_{k=0}^m G_k$ an elementary function of order $m$. Then we have $\Ebb\big[J_{n}^{(\al_1,\ldots, \al_{n})}(F)_{\,t}\,J_{m}^{(\bt_1, \ldots, \bt_m)}(G)_{\,t}\big|\Fscr_0\big]=0$, $t\in[0,T]$, if $n\neq m$, while, if $m=n$,
\beqa\nonumber
\lefteqn{
\Ebb\big[J_{n}^{(\al_1,\ldots, \al_{n})}(F)_{\,t}J_{n}^{(\bt_1, \ldots, \bt_n)}(G)_{\,t}\big|\Fscr_0\big]}\\
\label{eq:IRmix.el}&=&F_0\,G_0\displaystyle\int_0\p t\int_0\p {t_n-}\!\!\!\!\!\!\!\!\!\cdots\int_0\p{t_2-}\!\!\!\! F_1(t_1)G_1(t_1)\ldots F_n(t_n)G_n(t_n)\,\rmd\aPP{\X{\al_1}}{\X{\bt_1}}_{t_1}\ldots
\rmd\aPP{\X{\al_n}}{\X{\bt_n}}_{t_n}\,.
\eeqa
\end{lemma}
\begin{proof} We start proving (i). If $n=0$, $J_0(F)$ is obviously a (constant) bounded martingale and hence it belongs to $\Hscr^2$. If $n=1$, the statement follows from the properties of the stochastic integral because $(J_0(F_0)_{t-})_{t\in [0,T]}\in\Lrm\p2(\X{\al_1})$ and $J\p{(\al_1)}_1(F_0\otimes F_1)_0=J_0(F_0)_{0-}F_1(0)\,\X{\al_1}_0 =0$ in view of the setting $X_{0-}=\nolinebreak0$ for any c\`adl\`ag process $X$. 
Now we assume that the claim holds for $n$ and we verify it for $n+1$. From \eqref{eq:def.sim.mul.int} and the definition of the stochastic integral, it is shown as for $n=1$ that $J\p{(\al_1,\ldots,\al_{n+1})}_{n+1}(F)_0=0$. To see that $J\p{(\al_1,\ldots,\al_{n+1})}_{n+1}(F)$ is a square integrable martingale, we only need to verify that the integrand on the right-hand side of \eqref{eq:def.sim.mul.int} for $m=n+1$ is in $\Lrm\p2(\X{\al_{n+1}})$:
\[
\begin{split}
\Ebb\Big[&\int_0\p T \big(J\p{(\al_1,\ldots,\al_n)}_{n}(F_0\otimes\cdots\otimes F_n)_{u-}\,F_{n+1}(u)\big)\p2
\rmd\aPP{\X{\al_{n+1}}}{\X{\al_{n+1}}}_u\Big]\\
&\quad\leq\Ebb\big[(
J\p{(\al_1,\ldots,\al_n)}_{n}(F_0\otimes\cdots\otimes F_n)_{T})\p2\big]\int_0\p T (F_{n+1}(u))\p2\rmd\aPP{\X{\al_{n+1}}}{\X{\al_{n+1}}}_u<+\infty,
\end{split}
\]
where we used that the predictable covariation $\aPP{\X{\al_{n+1}}}{\X{\al_{n+1}}}$ and $F_{n+1}$ are deterministic and the induction hypothesis. Now we show (ii).
As a first step we assume that $m=n$ and we deduce the result by induction. If $n=0$ there is nothing to prove. Now we assume \eqref{eq:IRmix.el} for $n$ and verify it for $n+1$. Because from (i) follows that the elementary iterated integrals are in $\Hscr\p2_0$, this is an immediate consequence of the relation
\[\begin{split}
&\Ebb\big[J_{n+1}^{(\al_1,\ldots, \al_{n+1})}(F)_{\,t}\,J_{n+1}^{(\bt_1, \ldots, \bt_{n+1})}(G)_{\,t}|\Fscr_0\big]=\Ebb\big[\aPP{J_{n+1}^{(\al_1,\ldots, \al_{n+1})}(F)}{J_{n+1}^{(\bt_1, \ldots, \bt_{n+1})}(G)}_{\,t}|\Fscr_0\big]\\
&=\int_0\p t\Ebb\big[J_{n}^{(\al_1,\ldots, \al_{n})}(\dis{\otimes_{k=0}^n} F_k)_{\,u-}\,J_{n}^{(\bt_1, \ldots, \bt_{n})}(\dis{\otimes_{k=0}^n} G_k)_{\,u-}|\Fscr_0\big] F_{n+1}(u)G_{n+1}(u)\,\rmd
\aPP{\X{\al_{n+1}}}{\X{\bt_{n+1}}}_u
\end{split}
\]
(following from the properties of the predictable covariation of stochastic integrals and the definition of the iterated integrals) and the induction hypothesis. This completes the proof of (ii) for the case $n=m$. Now we consider the case $n\neq m$, say $n=m+p$, $p>0$. We proceed by induction on $m$. If $m=0$, then (ii) follows from (i). It remains to prove the statement for $m+1$ under the assumption that it is fulfilled for $m$. To this end, we notice that if $X,Y\in\Hscr\p{\,2}$ are such that $\Ebb[X_{\,t}Y_{\,t}|\Fscr_0]=0$ for every $t\in[0,T]$, then $\Ebb[X_{t-}Y_{t-}|\Fscr_0]=0$ for every $t\in[0,T]$. Indeed, as a consequence of Doob's inequality, $\sup_{t\in[0,T]}|X_tY_t|\le 1/2\,\sup_{t\in[0,T]}X_t^2+1/2\,\sup_{t\in[0,T]}Y_t^2$ is integrable and the conclusion follows from Lebesgue's theorem on dominated convergence. To complete the proof of the induction step, now we have only to recall (i) and that $\Ebb[X_{\,t}Y_{\,t}|\Fscr_0]=\Ebb[\aPP{X}{Y}_t|\Fscr_0]$ for every $X,Y\i
 n\Hscr\p{\,2}_0$, $t\in[0,T]
 $ and to apply the induction hypothesis.
\end{proof}

For $t\in[0,T]$ and $n\geq 1$, we introduce the sets
\begin{equation}\label{eq:def.Mn}
M_{\,t}^{(n)}:=\{(t_1,\ldots,t_n): \ 0\leq t_1\leq\ldots\leq t_n\leq t\},\qquad \ol{M}_{\,t}^{(n)}:=\Om\times M_{\,t}^{(n)}\,.
\end{equation}
For every $\al\in\Lm$, we also introduce the finite measure $m^{(\al)}$ on $([0,T],\Bscr([0,T]))$ generated by the right-continuous increasing function $\aP{\X{\al}}$. For any $\al_1,\ldots, \al_n\in\Lm$, we define  $m^{(\al_1,\ldots, \al_n)}_{\mb P}$ as the product measure $\mb P\otimes\bigotimes_{k=1}^n m^{(\al_k)}$ on $(\Om\times[0,T]^n,\mc F_0\times\Bscr([0,T]^n))$.

For $n\geq1$, we denote by $\Escr_{n,t}^{(\al_1, \ldots, \al_n)}$ the linear subspace of $L^2(\ol{M}_{\,t}^{(n)}, m^{(\al_1,\ldots, \al_n)}_{\mb P})$ generated by the elementary functions $F$ of order $n$ restricted to $\ol{M}_{\,t}^{(n)}$.  Applying the expectation to \eqref{eq:IRmix.el}, from the resulting isometry relation it easily follows that $J_n\p{(\al_1,\ldots,\al_n)}(\cdot)_t$ can be uniquely extended \textit{linearly} to $\Escr_{n,t}^{(\al_1, \ldots, \al_n)}$. Clearly, relation \eqref{eq:IRmix.el} extends to all $F\in\Escr_{n,t}^{(\al_1, \ldots, \al_n)}$ and $G\in\Escr_{n,t}^{(\bt_1, \ldots, \bt_n)}$. In particular, the mapping $J_n\p{(\al_1,\ldots,\al_n)}(\cdot)_t$ linearly extended to $\Escr_{n,t}^{(\al_1, \ldots, \al_n)}$ is a linear and isometric mapping from $L^2(\ol{M}_{\,t}^{(n)}\!\!, m^{(\al_1,\ldots, \al_n)}_{\mb P})$ into $L\p2(\Pbb)$.
Setting $\al_k=\bt_k$, $k=1,\ldots,n$, and $F=G\in\Escr_{n,t}^{(\al_1, \ldots, \al_n)}$ in the extended isometry relation \eqref{eq:IRmix.el} and then taking the expectation, yields
\begin{equation}\label{eq:isom}
\|J_n\p{(\al_1,\ldots,\al_n)}(F)_t\|\p2_{L\p2(\Pbb)}=
\|F\|^2_{L^2(\ol{M}_{\,t}^{(n)},\, m^{(\al_1,\ldots, \al_n)}_{\mb P})}\,.
\end{equation}
On the other side, the linear space $\Escr_{n,t}^{(\al_1, \ldots, \al_n)}$ is dense in $L^2(\ol{M}_{\,t}^{(n)}, m^{(\al_1,\ldots, \al_n)}_{\mb P})$ and therefore, the linear mapping
$F\mapsto J_{n}^{(\al_1,\ldots, \al_{n})}(F)_{\,t}$,\, $F\in \Escr_{n,t}^{(\al_1, \ldots, \al_n)}$,  can uniquely be extended to an isometry on the space $L^2(\ol{M}_{\,t}^{(n)}, m^{(\al_1,\ldots, \al_n)}_{\mb P})$ with values in $L^2(\Pbb)$, for every $n\geq1$ and $\al_1,\ldots,\al_n\in\Lm$. We denote this extension by $J_{n}^{(\al_1,\ldots, \al_{n})}(\cdot)_{\,t}$, $t\in[0,T]$. If $n=0$, then the iterated integrals $J_0(F_0)_t$ of order zero evaluated at time $t\in[0,T]$ are defined just as the identity $J_0(F_0)_t=F_0$ for $F_0\in L^2(\Om,\mc F_0,\mb P)$, the closure of the space of elementary functions of order zero.
\begin{definition}\label{def:mul.st.int}
Let $n\geq1$, $\al_1,\ldots,\al_n\in\Lm$ and $F\in L^2(\ol{M}_{\,t}^{(n)}, m^{(\al_1,\ldots, \al_n)}_{\mb P})$. We call the stochastic process $J_{n}^{(\al_1,\ldots, \al_{n})}(F) :=(J_{n}^{(\al_1,\ldots, \al_{n})}(F)_{\,t})_{t\in[0,T]}$  the $n$-fold iterated stochastic integral of $F$ with respect to the martingales $(\X{\al_1}$, $\X{\al_2}, \ldots, \X{\al_n})$. If $n=0$, we say that the constant square integrable martingale defined by $J_0(F_0)=(J_0(F_0)_t)_{t\in[0,T]}$ is the $0$-fold iterated integral of $F_0\in L^2(\Om,\mc F_0,\mb P)$.   
\end{definition}
Using the linearity and isometry of the iterated integral, we obtain the following straightforward extension of Lemma \ref{lem:isom.sim.mul.int} to arbitrary $F\in L^2(\ol{M}_{\,t}^{(n)}, m^{(\al_1,\ldots, \al_n)}_{\mb P})$ and $G\in L^2(\ol{M}_{\,t}^{(m)}, m^{(\bt_1,\ldots, \bt_m)}_{\mb P})$.
\begin{proposition}\label{prop:isom.sim.mul.int} Let $n\geq1$, $\al_1,\ldots,\al_n\in\Lm$ and $F\in L^2(\ol{M}_{\,t}^{(n)}, m^{(\al_1,\ldots, \al_n)}_{\mb P})$.

\textnormal{(i)} The iterated integral $J_{n}\p{(\al_1,\ldots,\al_n)}(F)$ belongs to $\Hscr\p2_0$ .

\textnormal{(ii)} Let moreover $m\geq1$, $\bt_1, \ldots, \bt_m\in \Lm$ and $G\in L^2(\ol{M}_{\,t}^{(m)}, m^{(\bt_1,\ldots, \bt_m)}_{\mb P})$. Then, for every $t\in[0,T]$, we have: If $n\neq m$, then $\Ebb\big[J_{n}^{(\al_1,\ldots, \al_{n})}(F)_{\,t}J_{m}^{(\bt_1, \ldots, \bt_m)}(G)_{\,t}\big|\Fscr_0\big]=0$, while, if $m=n$,
\beqa\label{IRmix}
\nonumber\lefteqn{\Ebb\big[J_{n}^{(\al_1,\ldots, \al_{n})}(F)_{\,t}\,J_{n}^{(\bt_1, \ldots, \bt_n)}(G)_{\,t}\big|\Fscr_0\big]}\\
&=&\displaystyle\Ebb\left[\int_0\p t\int\p {t_{n}-}_0\!\!\!\!\!\!\!\!\!\cdots\int_0\p{t_2-}\!\!\!\!
F(t_1,\ldots,t_n)\,G(t_1,\ldots,t_n)\,\rmd\aPP{\X{\al_1}}{\X{\bt_1}}_{t_1}\ldots
\,\rmd\aPP{\X{\al_n}}{\X{\bt_n}}_{t_n}\bigg|\Fscr_0\right]\,.
\eeqa
\end{proposition}
In the following definition we introduce some spaces of iterated stochastic integrals. 
\begin{definition}\label{def:sp.mul.int}
(i) Let $\Jscr_{0,0}$ be the space of bounded $\mc F_0$-measurable martingales (the $0$-fold elementary iterated integrals) and $\Jscr_{0}$ the space of square integrable $\mc F_0$-measurable martingales (the $0$-fold iterated integrals). To simplify the notation, we shall identify martingales $J_0(F_0)$ of $\Jscr_{0,0}$ (resp., $\Jscr_{0}$) with the bounded (resp., square integrable) $\mc F_0$-measurable random variable $F_0$.
  
(ii) Let $n\geq1$ and $\al_1, \ldots, \al_n\in\Lm$. By $\Jscr_{n}^{(\al_1,\ldots, \al_{n})}$ (resp., $\Jscr_{n,0}^{(\al_1,\ldots, \al_{n})}$) we denote the space of $n$-fold (resp., the linear hull of elementary) iterated stochastic integrals relative to the square integrable martingales $(\X{\al_1}, \X{\al_2}, \ldots, \X{\al_n})$.

(iii) For all $n\geq1$, we introduce
\begin{equation}\label{def:sp.mul.int.n}
\quad\scr J_{n,0}\!\!:=\!\!\Span\Big(\bigcup_{(\al_1,\ldots,\al_n)\in \Lm^n} \scr J_{n,0}^{(\al_1,\ldots, \al_{n})}\Big),\qquad\scr J_n\!\!:=\!\! \cl\Big(\Span\Big(\bigcup_{(\al_1,\ldots,\al_n)\in \Lm^n}\scr J_{n}^{(\al_1,\ldots, \al_{n})}\Big)\Big)_{\Hscr^2}\,,
\end{equation}
and then we define
\begin{equation}\label{def:sp.mul.in}
\scr J_{\rme}\!\!:=\!\!\Span\Big(\bigcup_{n\geq 0} \scr J_{n,0}\Big),\!\!\qquad\qquad\qquad\qquad
\scr J\!\!:=\!\!\cl\Big(\Span\Big(\bigcup_{n\geq 0} \scr J_{n}\Big)\Big)_{\Hscr^2}\,.
\end{equation}
We call $\Jscr$ the space of iterated integrals \emph{generated by} $\Xscr$.

(iv) By $\Jscr_T$ we denote the linear subspace of $L\p2(\Pbb)$ of terminal variables of iterated integrals from $\Jscr$. The linear spaces $\Jscr_{n,T}\p{(\al_1,\ldots,\al_n)}$ and $\Jscr_{n,T}$ of random variables in $L\p2(\Pbb)$ are introduced analogously from the spaces of processes $\Jscr_{n}\p{(\al_1,\ldots,\al_n)}$ and $\Jscr_n$, respectively, $n\geq0$.
\e{definition} 
Now we state the definition of the \emph{chaotic representation property} on the space $L\p2(\Pbb)$.

\begin{definition}\label{def:crp} We say that $\Xscr=\{\X{\al},\ \al\in\Lm\}$ possesses the \emph{chaotic representation property} (CRP) on the Hilbert space $L\p2(\Pbb)=L\p2(\Om,\Fscr,\Pbb)$ if the linear space $\Jscr_T$ (cf.\ Definition \ref{def:sp.mul.int} (iii)) is equal to $L\p2(\Pbb)$.
\end{definition}
We stress that, because the spaces $(L\p2(\Pbb),\|\cdot\|_2)$ and $(\Hscr\p{\,2},\|\cdot\|_{\Hscr\p{\,2}})$ can be identified, we can equivalently claim that $\Xscr$ possesses the CRP if $\Jscr=\Hscr\p{\,2}$.

Proposition \ref{prop:isom.sim.mul.int} (ii) yields that $\Jscr_{n,T}$ $(n\geq1)$ (resp., $\Jscr_{n}$ $(n\geq1)$) are pairwise orthogonal closed subspaces of $L\p2(\Pbb)$ (resp., $\Hscr\p{\,2}$). Furthermore, it can easily be checked that $\Jscr_{0,T}$ (resp., $\Jscr_{0}$) is orthogonal to $\Jscr_{n,T}$ (resp., $\Jscr_{n}$) for all $n\geq 1$. This immediately leads to the following equivalent description of the CRP.
\begin{proposition}\label{prop:equivCRP}
\emph{(i)} It holds $\Jscr_T=\bigoplus_{n=0}\p\infty\Jscr_{n,T}$ (resp., $\Jscr=\bigoplus_{n=0}\p\infty\Jscr_{n}$).

\emph{(ii)} The family $\Xscr$ possesses the CRP if and only if
$L\p2(\Pbb)=\bigoplus_{n=0}\p\infty\Jscr_{n,T}$ (resp., $\Hscr\p{\,2}=\bigoplus_{n=0}\p\infty\Jscr_{n}$).
\end{proposition}
Now we shortly discuss the relation between the CRP and the \emph{predictable representation property} (PRP).
We recall that a closed linear subspace $\Hscr$ of $\Hscr\p{\,2}$ is a stable subspace of $\Hscr\p2$ if  $1_AX\p\tau$ belongs to $\Hscr$, for every  stopping time $\tau$, $A\in\Fscr_0$ and $X\in\Hscr$.
Let $\Xscr$ be a subfamily of $\Hscr\p{\,2}$. The stable subspace generated by $\Xscr$ is denoted by $\Lscr\p2(\Xscr)$ and is defined as the smallest stable subspace of $\Hscr\p{\,2}$ containing $\Xscr$. Note that $\Lscr\p2(\Xscr)$ is the smallest stable subspace of $\Hscr\p{\,2}$ containing the set of stochastic integrals $\{H\cdot X,\ H\in\Lrm\p2(X),\ X\in\Xscr\}$. Furthermore we have $\Lscr\p2(\{1\})=\{X\in\Hscr\p2,\ X_{\,t}\equiv X_0\}=\Jscr_0$. For more details on the theory of stable subspaces of martingales, cf.\ \citet{J79}, Chapter IV. We say that $\Xscr$ possesses the PRP with respect to $\Fbb$ if $\Lscr\p2(\Xscr\cup\{1\})=\Hscr\p{\,2}(\Fbb)$.
We now assume that $\Xscr$ possesses the CRP. Clearly the inclusion
$\Jscr_\rme\subseteq\Lscr\p2(\Xscr\cup\{1\})$ 
holds. Indeed, an elementary iterated integral of order $n\geq1$ can always be regarded as a stochastic integral with respect to 
an element of $\Xscr$ (cf.\ Definition \ref{elem.it.int.}) and $\Jscr_{0,0}\subseteq\Lscr\p2(\{1\})$. Using that 
$\Lscr\p2(\Xscr\cup\{1\})$ is closed in $\Hscr\p{\,2}$ we obtain $\Hscr^{\,2}(\Fbb)=\Jscr =\cl(\Jscr_\rme)_{\Hscr\p{\,2}}\subseteq\Lscr\p2(\Xscr\cup\{1\})$. 
Hence $\Lscr\p2(\Xscr\cup\{1\})=\Hscr^{\,2}(\Fbb)$. Thus we have shown that, for every family $\Xscr\subseteq\Hscr\p{\,2}$ for which $\aPP{X}{Y}$ is deterministic, $X,Y\in\Xscr$, the CRP implies the PRP.

The following technical lemma will be used to prove Theorem \ref{thm:cl.st.it} below.
\begin{lemma}\label{con.mul.in}
Let $\X{\al_1},\ldots,\ \X{\al_m};\ \X{\bt_1},\ldots,\ \X{\bt_m};\ \X{\al_1\p n},\ldots,\ \X{\al_m\p n}\in\Xscr$ be such that for every $k=1,\ldots,m$, $\X{\al_k\p n}\longrightarrow\X{\al_k}$ in $\Hscr\p{\,2}$ as $n\rightarrow+\infty$. Then for any elementary function $F=1\otimes(\otimes_{k=1}^mF_k)$ of order $m$ we have:
\begin{equation}\label{eq:con.mul.in1}
\begin{split}\lim_{n\rightarrow+\infty}\textstyle\sup_{t\in[0,T]}
&\displaystyle\Big|\int_0\p t\int_0\p{t_m-}\!\!\!\!\!\!\!\cdots\int_0\p{t_2-}\!\!\!\! (\otimes_{k=1}^mF_k)(t_1,\ldots,t_m)
\,\rmd\aPP{\X{\al_1\p n}}{\X{\al_1\p n}}_{t_1}\ldots\ \rmd\aPP{\X{\al_m\p n}}{\X{\al_m\p n}}_{t_m}\\&-\int_0\p t\int_0\p{t_m-}\!\!\!\!\!\!\!\cdots\int_0\p{t_2-} (\otimes_{k=1}^mF_k)(t_1,\ldots,t_m)\,
\rmd\aPP{\X{\al_1}}{\X{\al_1}}_{t_1}\ldots\ \rmd\aPP{\X{\al_m}}{\X{\al_m}}_{t_m}\Big|=0
\end{split}
\end{equation}
and
\begin{equation}\label{eq:con.mul.in2}
\begin{split}\lim_{n\rightarrow+\infty}\textstyle\sup_{t\in[0,T]}
&\displaystyle\Big|\int_0\p t\int_0\p{t_m-}\!\!\!\!\!\!\!\cdots\int_0\p{t_2-}\!\!\!\!
 (\otimes_{k=1}^mF_k)(t_1,\ldots,t_m)
\,\rmd\aPP{\X{\al_1\p n}}{\X{\bt_1}}_{t_1}\ldots\ \rmd\aPP{\X{\al_m\p n }}{\X{\bt_m}}_{t_m}\\&-\int_0\p t\int_0\p{t_m-}\!\!\!\!\!\!\!\cdots\int_0\p{t_2-}\!\!\!\! (\otimes_{k=1}^mF_k)(t_1,\ldots,t_m)\,
\rmd\aPP{\X{\al_1}}{\X{\bt_1}}_{t_1}\ldots\ \rmd\aPP{\X{\al_m}}{\X{\bt_m}}_{t_m}\Big|=0\,.
\end{split}
\end{equation}
\end{lemma}
\begin{proof}
We verify only \eqref{eq:con.mul.in1} because \eqref{eq:con.mul.in2} easily follows from \eqref{eq:con.mul.in1} using the polarization formula 
\[\aPP{X}{Y}=\frac14(\aPP{X+Y}{X+Y}-\aPP{X-Y}{X-Y})\,,\quad X,Y\in\Hscr\p2\,,
\] 
and the linearity of the Riemann--Stieltjes integral with respect to the integrator. Now we start with the proof of \eqref{eq:con.mul.in1}. Let $F=1\otimes F_1\otimes\cdots \otimes F_m$ be such that $|F_k|\leq c$ for $k=1,\ldots,m$ with $c>0$, and let $t\in[0,T]$. We introduce the abbreviation
\begin{equation}\label{def:Hproc}
H\p{(\al_1,\ldots,\al_m)}_{t,m}:=\int_0\p t\int_0\p{t_m-}\!\!\!\!\!\!\!\cdots\int_0\p{t_2-}\!\!\!\! (\otimes_{k=1}^mF_k)(t_1, \ldots, t_m)\,\rmd\aPP{\X{\al_1}}{\X{\al_1}}_{t_1}\ldots\ \rmd
\aPP{\X{\al_m}}{\X{\al_m}}_{t_m}
\end{equation}
with the convention $H_{t,0}=1$. Note that $H\p{(\al_1,\ldots,\al_m)}_{t,m}=\int_0\p tF_m(u)H\p{(\al_1,\ldots,\al_{m-1})}_{u,m-1}
\,\rmd\aPP{\X{\al_m}}{\X{\al_m}}_u$.
Rewriting the left-hand side of \eqref{eq:con.mul.in1} using \eqref{def:Hproc} and observing that for all bounded measurable processes $H,K$ and for $X,Y$ in $\Hscr\p2$ the equality
\[
H\cdot\aPP{X}{X}-K\cdot\aPP{Y}{Y}=(H-K)\cdot\aPP{X}{Y}+H\cdot
\aPP{X}{X-Y}-K\cdot\aPP{Y}{Y-X}
\]
holds and that $\vert H\cdot A\vert\leq |H|\cdot\Var(A)$ for every $A\in\Vscr$ and measurable process $H$, we get
\[\begin{split}&\textstyle\sup_{t\in[0,T]}\displaystyle\Big|\int_0\p tF_m(t_m)H\p{(\al_1\p n,\ldots,\al_{m-1}\p n)}_{t_m-,m-1}\,\rmd\aPP{\X{\al_m\p n}}{\X{\al\p n_m}}_{t_m}-\int_0\p tF_m(t_m)H\p{(\al_1,\ldots,\al_{m-1})}_{t_m-,m-1}\,\rmd\aPP{\X{\al_m}}
{\X{\al_m}}_{t_m}\Big|\\
&\qquad\qquad\leq c\textstyle\sup_{t\in[0,T]}\displaystyle\big|H\p{(\al_1\p n,\ldots,\al_{m-1}\p n)}_{t,m-1}-H\p{(\al_1,\ldots,\al_{m-1})}_{t,m-1}\big|\|\X{\al\p n_m}\|_{\Hscr\p2}\|\X{\al_m}\|_{\Hscr\p2}
\\&\qquad\qquad\quad+c\textstyle\sup_{t\in[0,T]}\displaystyle
\big|H\p{(\al_1\p n,\ldots,\al_{m-1}\p n)}_{t,m-1}\big|\|\X{\al_m\p n}\|_{\Hscr\p2}\|\X{\al_m\p n}-\X{\al_m}\|_{\Hscr\p2}\\
&\qquad\qquad\quad+c\textstyle\sup_{t\in[0,T]}\displaystyle
\big|H\p{(\al_1,\ldots,\al_{m-1})}_{t,m-1}\big|
\|\X{\al_m}\|_{\Hscr\p2}\|\X{\al_m}-\X{\al_m\p n}\|_{\Hscr\p2}\,,
\end{split}\]
where in the last passage we used that $F_m$ is bounded, Kunita--Watanabe's inequality in the form of Meyer \cite{M76}, Corollary II.22, the relation $\Ebb[X_T\p2]\geq\Ebb[\aPP{X}{X}_T]$ for $X\in\Hscr\p2$, and the assumption that all the predictable covariations are deterministic. Because for $m=1$ the previous inequality becomes
\[\textstyle\sup_{t\in[0,T]}\displaystyle\Big| H_{t,1}\p{(\al_1\p n)}- H_{t,1}\p{(\al_1)}\Big|\leq c\|\X{\al_1\p n}-\X{\al_1}\|_{\Hscr\p{\,2}}\|\X{\al_1\p n}\|_{\Hscr\p{\,2}}+c\|\X{\al_1}\|_{\Hscr\p{\,2}}\|\X{\al_1\p n}-\X{\al_1}\|_{\Hscr\p{\,2}}\,,
\]
and the right-hand side converges to zero as $n\rightarrow+\infty$, \eqref{eq:con.mul.in1} follows by induction.
\end{proof}
Let $\Zscr$ be a subfamily of $\Xscr$. We denote by $\Jscr\p\Xscr$ and $\Jscr\p\Zscr$ the spaces of iterated integrals generated by $\Xscr$ and $\Zscr$, respectively.
\begin{theorem}\label{thm:cl.st.it}
If $\Xscr\subseteq\cl(\Span(\Zscr))_{\Hscr\p{\,2}}$, then $\Jscr\p\Zscr=\Jscr\p\Xscr$.
\end{theorem}
\begin{proof}
Because of $\Jscr\p\Zscr=\Jscr\p{\Span(\Zscr)}$, without loss of generality we can assume that $\Zscr$ is a linear space. Then,
for all $\X{\al_1},\ldots,\X{\al_m}\in\Xscr$, there exist $\X{\al_1\p n},\ldots,\X{\al_m\p n}\in\Zscr$ such that $\X{\al_k\p n}\longrightarrow\X{\al_k}$ in $\Hscr\p{\,2}$ as $n\rightarrow+\infty$, $k=1,\ldots,m$. Let now $F=F_0\otimes\cdots\otimes F_m$ be an elementary function. We show that $J_{m}\p{(\al_1,\ldots,\al_m)}(F)$ belongs to $\Jscr\p\Zscr$. Let $J_{m}\p{(\al\p n_1,\ldots,\al\p n_m)}(F)$ be the elementary iterated integral of $F$ with respect to $(\X{\al_1\p n},\ldots,\X{\al_m\p n})$. Because of \eqref{eq:IRmix.el}, for all $t\in[0,T]$ we have

\[
\begin{split}
&\Ebb\big[(J_{m}\p{(\al_1,\ldots,\al_m)}(F)_{\,t}-J_{m}\p{(\al\p n_1,\ldots,\al\p n_m)}(F)_{\,t})\p{\,2}\big]\\
&\quad=\Ebb\big[(J_{m}\p{(\al_1,\ldots,\al_m)}
(F)_{\,t})\p{\,2}\big]+\Ebb\big[(J_{m}\p{(\al\p n_1,\ldots,\al\p n_m)}(F)_{\,t})\p{\,2}\big]-2\Ebb\big[J_{m}\p{(\al_1,\ldots,\al_m)}
(F)_{\,t}J_{m}\p{(\al\p n_1,\ldots,\al\p n_m)}(F)_{\,t}\big]
\\&\quad=\Ebb[F_0^2]\Bigg(\int_0\p t\int_0\p{t_m-}\!\!\!\!\!\!\!\cdots\int_0\p{t_2-}
(\otimes_{k=1}^mF_k)\p2(t_1,\ldots,t_m)
\,\rmd\aPP{\X{\al_1}}{\X{\al_1}}_u\ldots\,\rmd
\aPP{\X{\al_m}}{\X{\al_m}}_u
\\&\hspace{2cm}+\int_0\p t\int_0\p{t_m-}\!\!\!\!\!\!\!\cdots\int_0\p{t_2-}\!\!\!\!
(\otimes_{k=1}^mF_k)\p2(t_1,\ldots,t_m)\,\rmd\aPP{\X{\al_1\p n}}{\X{\al_1\p n}}_u\ldots\,\rmd\aPP{\X{\al_m\p n}}{\X{\al_m\p n}}_u
\\&\hspace{2cm}-2\int_0\p t\int_0\p{t_m-}\!\!\!\!\!\!\!\cdots\int_0\p{t_2-}\!\!\!\!
(\otimes_{k=1}^mF_k)\p2(t_1,\ldots,t_m)\,\rmd\aPP{\X{\al_1\p n}}{\X{\al_1}}_u\ldots\,\rmd\aPP{\X{\al_m\p n}}{\X{\al_m}}_u\Bigg)\,,
\end{split}
\]
which converges to zero because of Lemma \ref{con.mul.in}. Since the space $\Jscr\p\Zscr$ is closed in $\Hscr^2$, we conclude that $J_{m}\p{(\al_1,\ldots,\al_m)}(F)\in\Jscr\p\Zscr$.
The result can be easily extended to arbitrary functions $F$ in the Hilbert space $L\p2(\ol{M}\p{(m)}_{\,T},m^{(\al_1,\ldots,\al_m)}_{\mb P})$ by linearity and isometry of iterated integrals. Finally, using the definition of $\Jscr\p\Xscr$, we get $\Jscr\p\Xscr\subseteq\Jscr\p\Zscr$. The converse inclusion is clear because $\Zscr\subseteq\Xscr$.
\end{proof}

Now we consider the case in which the martingales coming into play are orthogonal. We introduce the following notation: For every $\al_1,\ldots,\al_n;\ \bt_1,\ldots,\bt_m\in\Lm$ we write $(\al_1,\ldots,\al_n)\neq(\bt_1,\ldots,\bt_m)$ if $n\neq m$ or if $n=m$ there exists $1\leq\ell\leq n=m$ such that $\al_\ell\neq\bt_\ell$. The following proposition can be immediately deduced from Proposition \ref{prop:isom.sim.mul.int} (ii).
\begin{proposition}\label{prop:orth.sim.mul.int}
Let $n\geq1$ and $(\al_1,\ldots,\al_n),\ (\bt_1,\ldots,\bt_n)\in\Lm^n$. Suppose that there exists some index $i$ from $\{1,\ldots,n\}$ such that the martingales $\X{\al_i}, \X{\bt_{i}}$ are orthogonal, i.e., $\X{\al_i}_0\X{\bt_{i}}_0=0$ and $\aPP{\X{\al_i}}{\X{\bt_{i}}}=0$. Then the random variables $J_{n}\p{(\al_1,\ldots,\al_n)}(F)_{\,t}$ and $J_{n}\p{(\bt_1,\ldots,\bt_n)}(G)_{\,t}$ are orthogonal in $L\p2(\Pbb)$, for every $F\in L\p2(\ol{M}\p{(n)}_{\,T},m^{(\al_1,\ldots,\al_n)}_{\mb P})$ and $G\in L\p2(\ol{M}\p{(n)}_{\,T},m^{(\bt_1,\ldots,\bt_n)}_{\mb P})$, $t\in[0,T]$.
\end{proposition}
The following theorem will play an important role in the sequel.
\begin{theorem}\label{thm:mul.int.orth.sum}
Let $\Xscr:=\{\X{n},\ n\geq1\}\subseteq\Hscr\p2$ be a family consisting of countably many mutually orthogonal martingales such that $\aP{\X{n}}$ is deterministic for all $n\geq1$. Then the following identities hold:
\beq\label{eq:mul.int.orth.sum}
\Jscr=\Jscr_0\oplus\bigoplus_{n=1}\p\infty
\bigoplus_{(j_1,\ldots,j_n)\in\Nbb\p n} \Jscr_{n}\p{(j_1,\ldots,j_n)}\,,\qquad
\Jscr_T=\Jscr_{0,T}\oplus\bigoplus_{n=1}\p\infty
\bigoplus_{(j_1,\ldots,j_n)\in\Nbb\p n} \Jscr_{n,T}\p{(j_1,\ldots,j_n)}\,.
\eeq
\end{theorem}
\begin{proof}
We only verify the second relation. The space $\Jscr_{n,T}\p{(j_1,\ldots,j_n)}$ is closed in $L\p2(\Pbb)$ for every fixed $(j_1,\ldots,j_n)$ and $n\geq1$. If $(j_1,\ldots,j_n)\neq(i_1,\ldots,i_n)$, because of Proposition \ref{prop:orth.sim.mul.int} and the mutual orthogonality of the martingales in $\Xscr$, then $\Jscr_{n,T}\p{(j_1,\ldots,j_n)}$ and $\Jscr_{n,T}\p{(i_1,\ldots,i_n)}$ are orthogonal in $L\p2(\Pbb)$, $n\geq1$. For every fixed $n\geq1$, we put \[\Cscr\p{(n)}:=\bigoplus_{(j_1,\ldots,j_n)\in\Nbb\p n }\Jscr_{n,T}\p{(j_1,\ldots,j_n)}\,.\] Then $\Cscr\p{(n)}$ is closed because it is an orthogonal sum of countably many mutually orthogonal closed subspaces of $L\p2(\Pbb)$. Furthermore $\Cscr\p{(n)}$ contains $\Jscr_{n,T}\p{(j_1,\ldots,j_n)}$ for every $(j_1,\ldots,j_n)\in\Nbb\p n$ and hence also $\Jscr_{n,T}$ (cf.\ Definition \ref{def:sp.mul.int}). Conversely, from the definition of $\Jscr_{n,T}$ it is evident that the inclusion $\Cscr\p{(n)}\subseteq\nolinebreak\Jscr_{n,T}$ ho
 lds and, consequently,  $\Cscr\p{(n)}=\Jscr_{n,T}$, $n\geq1$. The statement of the theorem follows now from Proposition \ref{prop:equivCRP}.
\end{proof}
As a consequence of Theorem \ref{thm:mul.int.orth.sum} and Definition \ref{def:crp}, if the family $\Xscr:=\{\X{n},\ n\geq1\}$ consisting of countably many orthogonal martingales possesses the CRP on $L\p2(\Pbb)$, then the following orthogonal decompositions of $\Hscr\p{\,2}$ and $L\p2(\Pbb)$ hold:
\begin{equation}\label{eq:crp.or.sys}
\Hscr\p{\,2}=\Jscr_0\oplus\bigoplus_{n=1}\p\infty
\bigoplus_{(j_1,\ldots,j_n)\in\Nbb\p n} \Jscr_{n}\p{(j_1,\ldots,j_n)}\,,\qquad L\p2(\Pbb)=\Jscr_{0,T}\oplus\bigoplus_{n=1}\p\infty
\bigoplus_{{(j_1,\ldots,j_n)\in\Nbb\p n} }\Jscr_{n,T}\p{(j_1,\ldots,j_n)}\,.
\end{equation}

\section{Compensated-Covariation Stable Families and Iterated Integrals}\label{sec:ccs.f.mul.int}
We fix a time parameter $T>0$, a complete probability space $(\Om,\Fscr,\Pbb)$, a filtration $\Fbb=(\Fscr_{\,t})_{t\in[0,T]}$ satisfying the usual conditions and a family $\Xscr:=\{\X{\al},\ \al\in\Lm\}$ contained in $\Hscr\p{\,2}(\Fbb)$ indexed on the set $\Lm$. We recall that we always assume that $\aPP{\X{\al}}{\X{\bt}}$ is deterministic for all $\al,\bt\in\Lm$ without explicit mention.
In this section we study the properties of iterated integrals generated by $\Xscr$ under the further assumption that $\Xscr$ is a \emph{compensated-covariation stable} family of $\Hscr\p{\,2}$.

For $\al,\bt\in \Lm$ we define the process
\begin{equation}\label{def:compcov}
\N{\al,\bt}:=\cv{\N{\al}}{\N{\bt}}-\pb{\N{\al}}{\N{\bt}}
\end{equation}
which we call \emph{the compensated-covariation} process of $\N{\al}$ and $\N{\bt}$. The process $\pb{\N{\al}}{\N{\bt}}$ being the compensator of $[\N{\al},\N{\bt}]$, $\N{\al,\bt}$ is always a martingale and $\X{\al,\bt}_{\,0}=\X{\al}_{\,0}\X{\bt}_{\,0}$.
\begin{definition}\label{def:fam.Z}
(i) We say that the family $\Xscr:=\{\N{\al},\ \al\in \Lm\}\subseteq \Hscr\p{\,2}(\Fbb)$ is \emph{compensated-covariation stable} if for all $\al, \ \bt\in\Lm$ the compensated-covariation process $\N{\al,\bt}$ belongs to $\Xscr$.

(ii) Let $\Xscr$ be a compensated-covariation stable family and let $\al_1,\ldots,\al_m\in \Lm$ with $m\geq2$. The process $\N{\al_1,\ldots,\al_m}$ is defined recursively by
\begin{equation}\label{eq:it.com.cov}
\N{\al_1,\ldots,\al_m}:=[\N{\al_1,\ldots,\al_{m-1}},\N{\al_m}]
-\pb{\N{\al_1,\ldots,\al_{m-1}}}{\N{\al_m}}.
\end{equation}
\end{definition}
If $\Xscr$ is compensated-covariation stable, the process $\N{\al_1,\ldots,\al_m}$ belongs to $\Xscr$ for every $\al_1,\ldots,\al_m$ in $\Lm$ and $\N{\al_1,\ldots,\al_m}_{\,0}=\prod_{i=1}\p m \X{\al_i}_{\,0}$, $m\geq2$.

We begin with the following proposition. For the notations we refer to Section \ref{sec:def.it.int}. Suppose that $\Xscr:=\{\X{\al},\ \al\in\Lm\}$ is a compensated-covariation stable family of $\Hscr\p{\,2}$. We introduce the notation $\XX{\al}:=\X{\al}-\X{\al}_{\,0}$, for any $\al\in\Lm$. Note that $\XX{\al}\in\Hscr^2_0$, $\al\in\Lm$.
\bg{proposition}\label{MI} The stochastic integral $\XX{\al}_-\cdot M$ belongs to $\bigoplus_{k=0}\p{n+1}\scr J_{k,0}$, for every $\al\in\Lm$, $M\in \scr J_{n,0}$ and $n\geq0$.
\e{proposition}
\begin{proof} The proof will be given by induction on the order $n$ of the iterated integral $M\in\Jscr_{n,0}$. If $n=0$, i.e., $M\in\scr J_{0,0}$, then $\XX{\al}_-\cdot M\equiv 0\in\Jscr_{0,0}\subseteq \bigoplus_{k=0}\p{1}\scr J_{k,0}$. Let now $M\in\Jscr_{1,0}$. By linearity it suffices to take $M=J\p{(\al_1)}_1(F)$ where $F=F_0\otimes F_1$ is an elementary function of order $1$. Obviously, we have  $\XX{\al}=J^{(\al)}_1(1\otimes1)$ and $M=F_0\,(F_1\cdot \X{\al_1})$ from which it follows

\beqas\XX{\al}_-\cdot M_t&=&\int_0^t F_0\,J\p{(\al)}_1(1\otimes1)_{u-}\, F_1(u)\, \d \X{\al_1}_u\\
&=&J^{(\al,\al_1)}_2(F_0\otimes1\otimes F_1)\,.
\eeqas
This shows that $\XX{\al}_-\cdot M\in \Jscr_{2,0}\subseteq \bigoplus_{k=0}\p 2\Jscr_{k,0}$.
We now we fix $n\geq 2$ and assume that the statement is satisfied for all $M\in \scr J_{n,0}$. For the induction step it is enough to prove that $\XX{\al}_-\cdot M\in\bigoplus_{k=0}\p{n+2}\scr J_{k,0}$ for all $M$ from $\scr J_{n+1,0}^{(\al_1,\ldots, \al_{n+1})}$ and $\al_1,\ldots, \al_{n+1}\in \Lm$. To this end, let $M\in\scr J_{n+1,0}^{(\al_1,\ldots, \al_{n+1})}$ for some $\al_1,\ldots, \al_{n+1}$ from $\Lm$. In view of the linearity of $\scr J_{n+2,0}$ and of the iterated integral, without loss of generality, we can assume that $M$ is an \textit{elementary} iterated integral with respect to $(\X{\al_1}, \X{\al_2}, \ldots\ , \X{\al_{n+1}})$, i.e., $M$ has the representation
\begin{equation}\label{eq:rep.M}
M_{\,t}=J_{n+1}^{(\al_1,\ldots, \al_{n+1})}(F_0\otimes F_1\otimes\ldots\otimes F_{n+1})_{\,t}=\int_0^t N_{u-}\,F_{n+1}(u)\,\d {\X{\al_{n+1}}}_u,\quad t\in[0,T],
\end{equation}
where $N\in \scr J_{n,0}^{(\al_1,\ldots, \al_{n})}$ can  be written in the form
\beq\label{repN}
N_{\,t}=J_{n}^{(\al_1,\ldots, \al_{n})}(F_0\otimes\ldots\otimes F_{n})_{\,t} =\int_0^t R_{u-}\,F_{n}(u)\,\d {\X{\al_{n}}}_u,\quad t\in[0,T],
\eeq
with $R=J_{n-1}^{(\al_1,\ldots, \al_{n-1})}(F_0\otimes\ldots\otimes F_{n-1})\in \scr J_{n-1,0}^{(\al_1,\ldots, \al_{n-1})}$ and $F=F_0\otimes\ldots \otimes F_{n+1}$ is an elementary function of order $n+1$. Using partial integration for the product $\XX{\al}\,N$, the identities $N_0=N_{0-}=0$, $N_-\cdot \XX{\al}=N_-\cdot \X{\al}$ and $[\XX{\al},N]=[\X{\al},N]$ we get 
\beqa
\nonumber \XX{\al}_-\cdot M_{\,t}&=&\int_0^t \XX{\al}_{u-}\,N_{u-}\,F_{n+1}(u)\,\d {\X{\al_{n+1}}}_{u}\\
\label{firstterm}
&=&\int_0^t \Big(\int_0^{u-}N_{v-}\,\d \X{\al}_{v}\Big)\,F_{n+1}(u)\,\d {\X{\al_{n+1}}}_{u}\\
\label{secondterm}&&+\int_0^t \Big(\int_0^{u-}\XX{\al}_{v-}\,\d N_{v}\Big)\,F_{n+1}(u)\,\d {\X{\al_{n+1}}}_{u}\\
\label{thirdterm} &&+\int_0^t \big[\X{\al},N\big]_{u-}\,F_{n+1}(u)\,\d {\X{\al_{n+1}}}_{u}\,.
\eeqa
The term \eqref{firstterm} is equal to $J_{n+2,0}^{(\al_1,\ldots,\al_n, \al,\al_{n+1})}(F_0 \otimes\ldots\otimes F_{n}\otimes 1\otimes F_{n+1})_{\,t}$ and hence this process belongs to $\scr J_{n+2,0}^{(\al_1,\ldots,\al_n, \al,\al_{n+1})}\subseteq \bigoplus_{k=0}\p{n+2}\scr J_{k,0}$. The second term \eqref{secondterm} belongs to $\bigoplus_{k=0}\p{n+2}\scr J_{k,0}$ because, in view of the induction hypothesis, the integrand $\XX{\al}_-\cdot N$ belongs to $\bigoplus_{k=0}\p{n+1}\scr J_{k,0}$. Now we consider the third term \eqref{thirdterm}. Using the representation \eqref{repN} we can write
\beq\label{pointbrackets}
\big[ \X{\al},N\big]_{\,t}=\int_0^t R_{u-}\,F_n(u)\, \d \big[\X{\al_n},\X{\al}\big]_u,\quad
\aPP{\X{\al}}{N}_{\,t}=\int_0^t R_{u-}\,F_n(u)\, \d \aPP{\X{\al_n}}{\X{\al}}_u
\eeq
and hence, using the linearity of the stochastic integral and \eqref{eq:it.com.cov}, we get
\beq\label{inserting}
\big[ \X{\al},N\big]_{\,t}-\aPP{\X{\al}}{N}_{\,t}
=\int_0^t R_{u-}\,F_n(u) \,\d \X{\al_n,\al}_u\,.
\eeq
The third term \eqref{thirdterm} can be rewritten as
\beq
\begin{split}
\eqref{thirdterm}
=\nonumber\int_0^t \Big(\big[ \X{\al},N\big]_{u-}-\aPP{\X{\al}}{N}_{u-}\Big)\,F_{n+1}(u)\,\d {\X{\al_{n+1}}}_{u}
+\int_0^t \aPP{\X{\al}}{N}_{u-}\,F_{n+1}(u)\,\d {\X{\al_{n+1}}}_{u}
\end{split}
\eeq
and inserting \eqref{pointbrackets} and \eqref{inserting} in the previous equality we get that \eqref{thirdterm} is equal to
\beq\label{lastterm}
\int_0^t\!\! \int_0^{u-}\!\!\!\! R_{v-}\,F_n(v) \,\d \X{\al_n,\al}_v\,F_{n+1}(u)\,\d {\X{\al_{n+1}}}_{u}
+\int_0^t\!\! \int_0^{u-}\!\!\!\! R_{v-}\,F_n(v)\, \d \aPP{\X{\al_n}}{\X{\al}}_v\,F_{n+1}(u)\,\d {\X{\al_{n+1}}}_{u}\,.
\eeq
Since $R\in \scr J_{n-1,0}$, $F_n$ is bounded and, $\Xscr$ being compensated-covariation stable, $\X{\al_n,\al}\in\Xscr$, we can conclude that the right-hand side of \eqref{inserting} belongs to $\scr J_{n,0}$ and hence the first term in \eqref{lastterm} is an element of $\scr J_{n+1,0}\subseteq \scr \bigoplus_{k=0}\p{n+2}\scr J_{k,0}$.
Finally, for proving that the second term of \eqref{lastterm} belongs to $\bigoplus_{k=0}\p{n+2}\scr J_{k,0}$ we calculate its inner integral using partial integration:
\beqas
\lefteqn{\int_0^t R_{v-}\,F_n(v)\, \d \aPP{\X{\al_n}}{\X{\al}}_v}\\
&=&R_{\,t}\int_0^t F_n(s)\, \d \aPP{\X{\al_n}}{\X{\al}}_s-\int_0^t \int_0^{v-} F_n(s)\,\d \aPP{\X{\al_n}}{\X{\al}}_{s}\,\d R_v-\Big[R,\int_0^\cdot F_n(s)\,\d \aPP{\X{\al_n}}{\X{\al}}_s\Big]_{\,t}\\
&=&R_{\,t}\int_0^t F_n(s)\, \d \aPP{\X{\al_n}}{\X{\al}}_s-\int_0^t \int_0^{v-}F_n(s)\,\d \aPP{\X{\al_n}}{\X{\al}}_{s}\,\d R_v\\
&&-\int_0^t \big(\Delta\int_0^\cdot F_n(s)\,\d \aPP{\X{\al_n}}{\X{\al}}_s\big)_v\,\d R_v\\
&=&R_{\,t}\int_0^t F_n(s)\, \d \aPP{\X{\al_n}}{\X{\al}}_s-\int_0^t \int_0^{v}F_n(s)\,\d \aPP{\X{\al_n}}{\X{\al}}_{s}\,\d R_v
\eeqas
where in the last but one equality we have used \citet{JS00}, Proposition I.4.49 b). Substituting this in the second term of
\eqref{lastterm} we get
\beqa\label{eq:last}
\nonumber\lefteqn{\int_0^t \int_0^{u-} R_{v-}\,F_n(v)\, \d \aPP{\X{\al_n}}{\X{\al}}_v\,F_{n+1}(u)\,\d {\X{\al_{n+1}}}_{u}}\\
&=&\int_0^t R_{u-}\int_0^{u-} F_n(s)\, \d \aPP{\X{\al_n}}{\X{\al}}_s\,F_{n+1}(u)\,\d {\X{\al_{n+1}}}_{u}\\
\nonumber&&
-\int_0^t\int_0^{u-} \int_0^vF_n(s)\,\d \aPP{\X{\al_n}}{\X{\al}}_{s}\,\d R_v\,F_{n+1}(u)\,\d {\X{\al_{n+1}}}_{u}\,.
\eeqa
The first summand on the right hand side belongs to $\scr J_{n,0}$ because $R\in \scr J_{n-1,0}$ and the function $\wt{F}_{n+1}$ given by $\wt{F}_{n+1}(u):=\int_0^{u-} F_n(s)\, \d \aPP{\X{\al_n}}{\X{\al}}_s\,F_{n+1}(u)$ is bounded. 
From $R=J_{n-1,0}^{(\al_1,\ldots, \al_{n-1})}(F_0\otimes\ldots\otimes\nolinebreak F_{n-1})$ and the fact that the function $F$ with $F(v):=\int_0^vF_n(s)\,\d \aPP{\X{\al_n}}{\X{\al}}_{s}$ is bounded, we similarly obtain
\beq
\nonumber\int_0^{u} \int_0^v F_n(s)\,\d \aPP{\X{\al_n}}{\X{\al}}_{s}\,\d R_v=\int_0^{u} F(v)\,\d R_v
=J_{n-1,0}^{(\al_1,\ldots, \al_{n-1})}(F_0\otimes\ldots\otimes (FF_{n-1}))_u
\eeq
and hence the second integral of the right-hand side of \eqref{eq:last} is equal to
\beqs
\int_0^t J_{n-1,0}^{(\al_1,\ldots, \al_{n-1})}(F_0\otimes\ldots\otimes (FF_{n-1}))_{u-} \,F_{n+1}(u)\,\d \X{\al_{n+1}}_u
=J_{n,0}^{(\al_1,\ldots, \al_{n-1},\al_{n+1})}(F_0\otimes\ldots\otimes (FF_{n-1})\otimes F_{n+1})_{\,t}
\eeqs
which belongs to $\scr J_{n,0}\subseteq \bigoplus_{k=0}\p{n+2}\scr J_{k,0}$. The proof of the proposition is finished.
\end{proof}
Now we come to the the main result of this section. Recall that $\XX{\al}:=\X{\al}-\X{\al}_{\,0}$, $\al\in\Lm$.
\bg{theorem}\label{MItheorem}
Let $\scr X:=\{\X{\al},\ \al\in\Lm\}\subseteq\scr H\p{\,2}(\Fbb)$ be a compensated-covaria\-tion stable family. Then the stochastic integral $\big(\prod_{i=1}^m\XX{\al_i}_-\big)\cdot M$ belongs to $\bigoplus_{k=0}\p{n+m}\scr J_{k,0}$, for all $m\geq0$, $\al_1,\ldots,\al_m\in \Lm$ and martingales $M\in \bigoplus_{k=0}\p{n}\scr J_{k,0}$, for every $n\geq0$.
\e{theorem}
\begin{proof} The proof will be given by induction on $m$. If $m=0$, because $\prod_{i=1}\p0\XX{\al_i}_-:=1$ by convention, then the claim is evident because $M=1\cdot M$. We now assume that the statement of the theorem holds for $m\geq1$ and $M\in\bigoplus_{k=0}\p{n}\scr J_{k,0}$, for any $n\geq0$, and prove it for $m+1$. Let $\al_1,\dots,\al_{m+1}$ be given. Setting $\wt{M}:=\XX{\al_{m+1}}_-\cdot M$, we can calculate
\beqs
\big(\prod_{i=1}^{m+1}\XX{\al_i}_-\big)\cdot M =\big(\prod_{i=1}^{m}\XX{\al_i}_-\big)\cdot \big(\XX{\al_{m+1}}_-\cdot M\big)
= \big(\prod_{i=1}^{m}\XX{\al_i}_-\big)\cdot \wt{M}\,.
\eeqs
From Proposition \ref{MI} we obtain that $\wt{M}\in\bigoplus_{k=0}\p{n+1}\scr J_{k,0}$. The induction hypothesis yields that the right hand side belongs to $\bigoplus_{k=0}\p{(n+1)+m}\scr J_{k,0}=\bigoplus_{k=0}\p{n+(m+1)}\scr J_{k,0}$. This proves the induction step and hence the proof of the theorem is complete.
\end{proof}
An immediate consequence of Theorem \ref{MItheorem} is the following corollary:
\bg{corollary}\label{CMI}
Let $\scr X:=\{\X{\al},\ \al\in\Lm\}\subseteq\scr H\p{\,2}(\Fbb)$ be a compensated-covaria\-tion stable family. Then the process
$F_0\big(\prod_{i=1}^m\XX{\al_i}_-\big)\cdot \X{\al}$ belongs to $\bigoplus_{k=0}\p{n+1}\scr J_{k,0}$ for all bounded $\mc F_0$-measurable $F_0$ and parameters $\al, \al_1,\ldots,\al_m\in \Lm$, $m\geq1$.
\e{corollary}
\vspace{-15pt}
\begin{proof} We have $(F_0\prod_{i=1}\p m\XX{\al_i}_{-})\cdot\X{\al}=(\prod_{i=1}\p m\XX{\al_i}_{-})\cdot(F_0\,\XX{\al})$
and we can apply Theorem 4.3 to the martingale $M=J^{(\al)}_1(F_0\otimes1)=F_0\,\XX{\al}$.
\end{proof}
We conclude this section with the next corollary which shows that if $\Xscr=\{\X{\al},\ \al\in\Lm\}$ is a subfamily of $\Hscr\p{\,2}$ satisfying the assumptions of Theorem \ref{MItheorem} and another technical condition, then the random variable $\XX{\al}_{\,t}$ has finite absolute moments of every order for all $\al\in\Lm$ and $t\in[0,T)$.
\bg{corollary}
Let $\scr X=\{\X{\al},\ \al\in\Lm\}\subseteq\scr H\p{\,2}$ be a compensated-covaria\-tion stable family. If there exists $\bt\in \Lm$ such that $\aP{\X{\bt}}_{\,t}<\aP{\X{\bt}}_{\,T}$ for all $t<T$ then, for every $\al\in \Lm$ and $t<T$, $\XX{\al}_{\,t}$ has finite absolute moments of arbitrary order. 
\e{corollary}
\begin{proof} From Corollary \ref{CMI} we have $X:=(\XX{\al}_-)\p{\,m}\cdot \X{\bt}\in\bigoplus_{k=0}\p{m+1}\Jscr_{k,0}$. This implies that $X$ belongs to $\Hscr\p{\,2}$ and
$\aP{X}_{\,T}=\int_0^T \vert\XX{\al}_{u-}\vert^{2m}\, \d \aP{\X{\bt}}_u$, therefore
\[
\int_0^T \EE\big[\vert\XX{\al}_{u-}\vert^{2m}\big]\, \d \aP{\X{\bt}}_u=\EE\big[\aP{X}_{\,T}\big]<+\infty\,.
\]
Using the martingale property of $(\XX{\al},\mb F)$ and Jensen's inequality for conditional expectations, for every $t<u$, we can estimate
\[
\big|\XX{\al}_{\,t}\big|^{2m}=\big|\EE\big[\XX{\al}_u|\scr F_{\,t}\big]\big|^{2m}
=\big|\EE\big[\EE\big[\XX{\al}_u|\scr F_{u-}\big]|\scr F_{\,t}\big]\big|^{2m}
=\big|\EE\big[\XX{\al}_{u-}|\scr F_{\,t}\big]\big|^{2m}
\leq\EE\big[\big|\XX{\al}_{u-}\big|^{2m}|\scr F_{\,t}\big]
\]
which yields
$\EE\big[\big|\XX{\al}_{\,t}\big|^{2m}\big]\leq \EE\big[\big|\XX{\al}_{u-}\big|^{2m}\big],\quad 0\leq t<u\,.$
Hence we obtain
\[\big(\aP{\X{\bt}}_{\,T}-\aP{\X{\bt}}_{\,t}\big)\,\EE\big[\big| \XX{\al}_{\,t}\big|^{2m}\big]
\leq\int_{(t,T]}\EE\big[\big|\XX{\al}_{u-}\big|^{2m}\big]\,\d \aP{\X{\bt}}_u <+\infty\,.
\]
Let $\bt$ be chosen such that $\aP{\X{\bt}}_{\,t}<\aP{\X{\bt}}_{\,T}$ for all $t<T$. Then the above inequality yields
$\EE\big[\big|\XX{\al}_{\,t}\big|^{2m}\big]<+\infty,\ t<T$, which proves the claim.
\end{proof}

\section{The Chaotic Representation Property}\label{sec:chaos}
In this section we shall give sufficient conditions for a subfamily $\Xscr$ of $\Hscr\p{\,2}$ to possess the CRP.  For this purpose it will be useful to work with families of martingales in $\Hscr\p{\,2}$ which are stable under stopping with respect to deterministic stopping times.
For a given $\Xscr:=\{\X{\al},\ \al\in\Lm\}\subseteq\Hscr\p{\,2}$ and a collection $\Sscr$ of \emph{finite-valued} stopping times we define the family $\Xscr\p\Sscr$ by
\begin{equation}\label{def:stop.com.cov.fam}
\Xscr\p\Sscr:=\{{\N{\al}}\p\tau,\ \tau\in\Sscr,\ \al\in \Lm\},
\end{equation}
where the superscript $\tau$ denotes the operation of stopping at $\tau\in\Sscr$. It is clear that $\Xscr\p\Sscr\subseteq\Hscr\p{\,2}$. The following lemma states a condition on $\Sscr$ ensuring that $\Xscr\p\Sscr$ is a compensated-covariation stable family whenever $\Xscr$ is one. We will be particularly interested in the case $\Sscr=\Rbb_+$. Using the properties of the brackets $[\cdot,\cdot]$ and $\pb{\cdot}{\cdot}$ the proof is straightforward and therefore omitted. 
\begin{lemma}\label{lem:def:stop.com.cov.fam}
If $\Sscr$ is a minimum-stable family of finite-valued stopping times and the family $\Xscr\subseteq\Hscr\p{\,2}$ is compensated-covariation stable, then $\Xscr\p\Sscr$ is compensated-covariation stable, too.
\end{lemma}

Now we come to a useful representation formula for products of elements of a compensated-covariation stable family $\Xscr$.
Its proof is given by induction using integration by parts.
See \citet{DE13}, Proposition 3.3, where $X_0=0$ for every $X\in\Xscr$ was additionally assumed. However, using our convention $X_{0-}=0$ and the definition \eqref{def:cov} of the quadratic covariation $[X,Y]$ (including the jump $\Delta X_0\Delta Y_0=X_0Y_0$ at time zero), in case of a general compensated covariation stable family $\Xscr\subseteq\Hscr^2$, the reader may notice that the formula and its proof are completely the same.

\begin{proposition}\label{prop:rep.pol}
Let $\Xscr:=\{\N{\al},\ \al\in \Lm\}$ be a compensated-covariation stable family of $\Hscr\p{\,2}$. For every $m\geq1$ and $\al_{1},\ldots,\al_{m}\in \Lm$, we have
\begin{equation}\label{eq:rep.pol}
\begin{split}
\prod_{i=1}\p m\N{\al_i}
=&
\sum_{i=1}\p m\sum_{{1\leq j_1<\ldots<j_i\leq m}}\Bigl(\prod_{\begin{subarray}{c}k=1\\k\neq j_1,\ldots,j_i\end{subarray}}\p m \N{\al_k}_-\Bigr)\cdot\N{\al_{j_1},\ldots,\al_{j_i}}
\\
&+
\sum_{p=0}\p{m-2}\sum_{i=p+2}\p m\sum_{1\leq j_1<\ldots<j_i\leq m}\Bigl(\prod_{\begin{subarray}{c}k=1\\k\neq j_1,\ldots,j_i\end{subarray}}\p m\N{\al_k}_-
\prod_{\ell=i-p+1}\p{i} \Delta\N{\al_{j_\ell}}
\Bigr)\cdot
\pb{\N{\al_{j_1},\ldots,\al_{j_{i-p-1}}}}{\N{\al_{j_{i-p}}}}.
\end{split}
\end{equation}
\end{proposition}
Note that the initial value of the left-hand side of \eqref{eq:rep.pol} is given by the initial value of the first term on the right-hand side for $i=m$.

A simplification of formula \eqref{eq:rep.pol} can be obtained by assuming that the family $\Xscr$ consists of \emph{quasi-left continuous} martingales. Indeed, in this case one can choose continuous versions of the processes $\pb{\N{\al}}{\N{\bt}}$, $\al,\bt\in \Lm$ (cf.\ \citet{JS00}, Theorem I.4.2), and so all the terms appearing in the second sum on the right-hand side of \eqref{eq:rep.pol} vanish for $p\neq0$.

Let $\Xscr:=\{\X{\al},\ \al\in\Lm\}$ be a subfamily of $\Hscr\p{\,2}$. The filtration $\Fbb\p\Xscr:=(\Fscr_{t}\p\Xscr)_{t\in[0,T]}$ is defined as the smallest filtration satisfying the usual conditions and with respect to which each process in $\Xscr$ is adapted.
In the remaining part of the present paper we shall consider $\Xscr$ as a subfamily of $\Hscr\p{\,2}(\Fbb\p\Xscr)$ on the probability space $(\Om,\Fscr\p\Xscr_T,\Pbb)$.

We define the family $\Kscr$ by
\begin{equation}\label{def:fam.mon}
\textstyle\Kscr:=\big\{\prod_{i=1}\p m\N{\al_i}_{\,t_i}, \ \al_i\in \Lm,\ t_i\in[0,T],\ i=1,\ldots,m; \ m\geq0\big\}
\end{equation}
which is the family of monomials formed by products of elements of $\Xscr$ at different times. Obviously, $\sigma(\Kscr)$ augmented by the $\Pbb$-null sets of $\Fscr\p\Xscr_T$ equals $\Fscr\p\Xscr_T$.
We make the following assumption:
\begin{assumption}\label{ass:tot.pol}
The family $\Kscr$ defined in \eqref{def:fam.mon} is contained in $L\p{\,2}(\Om,\Fscr\p\Xscr_T,\Pbb)$ and is total (i.e., its linear hull is dense) in $L\p{\,2}(\Om,\Fscr\p\Xscr_T,\Pbb)$.
\end{assumption}
If $\Xscr=\{\X{\al},\ \al\in\Lm\}$ satisfies Assumption \ref{ass:tot.pol}, then $\N{\al}_{\,t}$ admits finite moments of every order for all $t\in[0,T]$ and $\al\in \Lm$, i.e.,
\begin{equation}\label{ass:ex.mom}
\Ebb[|\N{\al}_{\,t}|\p m]<+\infty,\quad t\in[0,T],\quad \al\in \Lm,\quad m\geq1.
\end{equation}
Sufficient conditions on the family $\Xscr$ for $\Kscr$ to fulfil Assumption \ref{ass:tot.pol} are extensively studied in the literature. The following well-known result, being useful for many applications, exploits the existence of finite exponential moments. For an elementary proof cf.\ \citet{DE13}, Theorem A.4.
\begin{theorem}\label{thm:den.pol}
If for every $\al\in \Lm$ and $t\geq0$ there exists a constant $c_\al(t)>0$ such that the expectation
$\Ebb[\exp(c_\al(t)|\N{\al}_{\,t}|)]$ is finite, then $\Kscr$ satisfies Assumption \ref{ass:tot.pol}.
\end{theorem}

Now we state two technical lemmas which will be needed in the proof of Proposition \ref{prop:tot.term.rv.int} below.
\begin{lemma}\label{lem:int.prod.Z}
Let $A$ be a deterministic process of finite variation, $\Xscr:=\{\N{\al},\ \al\in \Lm\}$ a family of martingales contained in $\Hscr\p{\,2}$ satisfying \eqref{ass:ex.mom} and $p,q\geq0$. We define the processes $K$ by
\begin{equation}\label{def:proc.H}
K:=\prod_{j=1}\p q\Delta\N{\bt_{j}}\prod_{i=1}\p p\N{\al_i}_{-},\quad \al_i,\bt_j\in \Lm,\quad i=1,\ldots,p;\ j=1,\ldots,q.
\end{equation}
Then the process $K\cdot A$ is of integrable variation.
\end{lemma}
\begin{proof}
Obviously,
$\vert\Var(K\cdot A_T)\vert\le \sup_{0\le t\le T} \vert K_t\vert\, \Var(A)_T$.
We will show below that $\sup_{0\le t\le T} \vert K_t\vert$ is integrable and, because $\Var(A)_T$ is deterministic and finite, this yields that $\Var(K\cdot A_T)$ is integrable and hence the claim.
For proving that $\sup_{0\le t\le T} \vert K_t\vert$ is integrable, we estimate
\bg{eqnarray}
\nonumber\sup_{0\le t\le T}\vert K_t\vert&\leq& \prod_{j=1}^q 2\sup_{0\le t\le T}\vert\N{\bt_{j}}_t\vert\,\;\;\prod_{i=1}^p\sup_{0\le t\le T}\vert\N{\al_{i}}_{t}\vert\\
\nonumber&\le&2^{q-1}\,\Big(\prod_{j=1}^q \sup_{0\le t\le T}\vert\N{\bt_{j}}_t\vert^{2}+\prod_{i=1}^p\sup_{0\le t\le T}\vert\N{\al_{i}}_{t}\vert^{2}\Big)\\
\label{Doob}&\le&2^{q-1}\Big(q^{-1}\sum_{j=1}^q\sup_{0\le t\le T}\vert\N{\bt_{j}}_t\vert^{2q}+p^{-1}\sum_{i=1}^p\sup_{0\le t\le T}\vert\N{\al_{i}}_{t}\vert^{2p}\Big)
\e{eqnarray}
where we have used the inequality $\prod_{k=1}^m a_k\leq m^{-1}\,\sum_{k=1}^m a_k^m$ for all nonnegative numbers $a_1, \ldots, a_m$ and $m\in\Nbb$. Now from \eqref{ass:ex.mom} it follows that $\vert\X{\al}\vert^m$ is a nonnegative submartingale for all $\al\in\Lm$ and $m\in\Nbb$. Using Doob's inequality we can conclude that the right-hand side of \eqref{Doob} is integrable which completes the proof.
\end{proof}

\begin{lemma}\label{lem:orth.lef.lim.prod}
Let $\Xscr=\{\X{\al},\ \al\in \Lm\}\subseteq\Hscr\p{\,2}$ be a family of martingales which satisfies \eqref{ass:ex.mom}; $q,r\geq1$ be fixed and $\xi\in L\p{\,2}(\Pbb)$ be such that, for some $\al_1,\ldots,\al_{q+r}\in \Lm$,  $\xi$ is orthogonal in $L\p{\,2}(\Pbb)$ to $\prod_{k=1}\p{q+r}\N{\al_k}_{t_k}$ for every $t_1,\ldots,t_{q+r}\in[0,T]$. Then $\xi$ is orthogonal to $\prod_{k=1}\p q\N{\al_{k}}_{t-}\prod_{j=q+1}\p {q+r}\N{\al_j}_{t}$ for every $t\in[0,T]$.
\end{lemma}
\bg{proof}
Fixing $t\in[0,T]$, similar as in the proof of Lemma \ref{lem:int.prod.Z} above we can show that
$$
\vert\xi\vert\prod_{k=1}\p q\sup_{0\le u\le T} \vert\N{\al_{k}}_{u}\vert\;\prod_{j=q+1}\p {q+r} \vert\N{\al_j}_{t}\vert
$$
is integrable. Choosing $t_n$ such that $t_n<t$ for all $n\in\Nbb$ and $t_n\rightarrow t$ as $n\rightarrow\infty$, we obtain, for all $n\in\Nbb$, $\Ebb[\xi\prod_{k=1}\p q\N{\al_{k}}_{t_n}\prod_{j=q+1}\p {q+r}\N{\al_j}_{t}]=0$ in view of the assumption and letting $n\rightarrow\infty$ the claim follows from Lebesgue's theorem on dominated convergence.
\e{proof}
Let $\Xscr:=\{\X{\al},\ \al\in \Lm\}$ be a compensated-covariation stable family of $\Hscr\p2$ such that, for every $\al,\bt\in \Lm$, the process $\pb{\X{\al}}{\X{\bt}}$ is deterministic. We introduce the following systems:
\begin{equation}\label{def:term.rv.int}
\begin{aligned}
&\Rscr:=\{F_0\,\textstyle(\dis{\prod_{i=1}\p m} \XX{\al_i}_{-})\cdot\X{\al},\quad \al,\al_1,\ldots,\al_m\in \Lm,\ m\geq0,\ F_0\in\Jscr_{0,0}\}\cup\Jscr_{0,0}\,,\\[2mm]
&\Rscr_T:=\{F_0\,\textstyle(\dis{\prod_{i=1}\p m} \XX{\al_i}_{-})\cdot\X{\al}_T,\quad \al,\al_1,\ldots,\al_m\in \Lm,\ m\geq0,\ F_0\in\Jscr_{0,0}\}\cup\Jscr_{0,0}\,.
\end{aligned}
\end{equation}
We stress that $\Rscr$ is contained in $\Hscr^2$ and $\Rscr_{T}$ in $L\p2(\Om,\Fscr\p\Xscr_T,\Pbb)$ and that $\Xscr\subseteq\cl(\Span(\Rscr))_{\Hscr\p{\,2}}$, hence $\Xscr$ and $\Rscr$ generate the same filtration. 

The next elementary identity will be useful in the proof of Proposition \ref{prop:tot.term.rv.int} below. For real numbers $a_r,b_r$, $r=1,\ldots,m$, we have:
\begin{equation}\label{eq:prod.n.bin}
 \prod_{r=1}\p m(a_r+b_r)=\sum_{r=0}\p m\sum_{1\leq q_1<\ldots<q_r\leq m}\prod_{\begin{subarray}{c}k=1\\k\neq q_1,\ldots,q_r\end{subarray}}\p {m}\!\!\!\!\!\!\!\! a_k\prod_{\ell=1}\p r b_{q_\ell}\,.
\end{equation}

The following proposition being used for the proof of Theorem \ref{thm:crp.ccs.fam}, is of interest in its own right.
\begin{proposition}\label{prop:tot.term.rv.int}
Suppose that $\Xscr$ is stable under stopping with respect to deterministic stopping times and that Assumption \ref{ass:tot.pol} holds. Then $\Rscr$ is total in $\Hscr^2$ and $\Rscr_T$ is total in $L\p2(\Om,\Fscr\p\Xscr_T,\Pbb)$.
\end{proposition}
\begin{proof} It is sufficient to verify the second claim.
We are going to show that any $\xi\in L\p2(\Om,\Fscr\p\Xscr_T,\Pbb)$ which is orthogonal to $\Rscr_T$ is orthogonal in $L\p2(\Om,\Fscr\p\Xscr_T,\Pbb)$ to $\prod_{i=1}\p m\X{\al_i}_{T}$ for every $\al_1,\ldots,\al_m\in \Lm$ and $m\geq0$. The stability under stopping with respect to deterministic stopping times of the family $\Xscr$ then yields that $\xi$ is also orthogonal to $\Kscr$, where $\Kscr$ is given in \eqref{def:fam.mon}. But by Assumption \ref{ass:tot.pol} the system $\Kscr$ is total in $L\p2(\Om,\Fscr\p\Xscr_T,\Pbb)$ and therefore $\xi$ is evanescent. This implies that $\Rscr_T\subseteq L\p2(\Om,\Fscr\p\Xscr_T,\Pbb)$ is a total subset and therefore the claim of the proposition.

Let $\xi\in L\p2(\Om,\Fscr\p\Xscr_T,\Pbb)$ be orthogonal to $\Rscr_T$. For verifying that $\xi$ is orthogonal to $\prod_{i=1}\p m\X{\al_i}_{T}$ for every $\al_1,\ldots,\al_m\in \Lm$ and $m\in\Nbb\cup\{0\}$, we proceed by strong induction on $m$. For $m=0$, we have $\prod_{i=1}\p 0\X{\al_i}_{T}=1\in\Rscr_T$ and hence by assumption $\xi$ is orthogonal to $\prod_{i=1}\p 0\X{\al_i}_{T}$. Let us now assume that for some fixed $m\geq1$
\begin{equation}\label{eq:thm.com.cov.fam.prp.in.hyp}
\Ebb\Bigl[\xi\prod_{i=1}\p n\N{\al_i}_{\,T}\Bigr]=0,\quad \al_1,\ldots,\al_n\in \Lm, \quad n\leq m \ .
\end{equation}
From \eqref{eq:thm.com.cov.fam.prp.in.hyp} and the property that $\Xscr$ is stable under stopping with respect to deterministic stopping times, it easily follows that
\begin{equation}\label{eq:orth.mon.in.hyp}
\Ebb\Bigl[\xi\prod_{i=1}\p n\N{\al_i}_{t_i}\Bigr]=0,\quad t_1,\ldots,t_n\in[0,T],\quad \al_1,\ldots,\al_n\in \Lm,\quad n\leq m.
\end{equation}
We now show that $\Ebb[\xi\prod_{i=1}\p{m+1}\X{\al_i}_{\,T}]=0$ for all $\al_1,\ldots, \al_{m+1}\in\Lm$. Representing $\prod_{i=1}\p{m+1}\X{\al_i}_{\,T}$ by the product formula \eqref{eq:rep.pol}, we get
\beqa\label{eq:rep.prod.orth}
\nonumber\;\;\;\Ebb\Bigl[\xi\!\!\!\prod_{i=1}\p {m+1}\N{\al_i}_{\,T}\Bigr]
\!\!\!\!\!\!&=&\!\!\!\!\!\!\!\sum_{i=1}\p{m+1}\sum_{1\leq j_1<\ldots<j_i\leq m+1}\!\!\!\!\!\!\Ebb\bigg[\!\xi\!\!\bigg(\!\!\Bigl(\!\!\!\!\prod_{\begin{subarray}{c}k=1\\k\neq j_1,\ldots,j_i\end{subarray}}\p {m+1}\!\!\!\!\!\!\!\!\N{\al_k}_-\Bigr)\cdot\X{\al_{j_1},\ldots,\al_{j_i}}_T\bigg)\bigg]\\
&\!\!\!\!\!\!\!\!\!\!\!\!\!\!\!\!\!\!\!\!\!\!\!\!\!\!\!\!\!\!\!\!\!\!\!\!\!\!\!\!\!\!\!\!\!\!\!\!\!\!\!\!\!\!\!\!\!\!\!\!\!\!\!+&\!\!\!\!\!\!\!\!\!\!\!\!\!\!\!\!\!\!\!\!\!\!\!\!\!\!\!\!\!\!\!\!\!\!\!\sum_{p=0}\p{m-1}\!\sum_{i=p+2}\p{m+1}\sum_{1\leq j_1<\ldots<j_i\leq m+1}\!\!\!\!\!\!\Ebb\bigg[\!\xi\!\!\bigg(\!\!\Bigl(\!\!\!\!\prod_{\begin{subarray}{c}k=1\\k\neq j_1,\ldots,j_i\end{subarray}}\p {m+1}\!\!\!\!\!\!\!\!\N{\al_k}_-\!\!\!\!\!
\prod_{\ell=i-p+1}\p{i}\!\!\!\!\!\!\! \Delta\N{\al_{j_\ell}}\!\Bigr)\cdot
\pb{\N{\al_{j_1},\ldots,\al_{j_{i-p-1}}}}{\N{\al_{j_{i-p}}}}_{\,T}\!\!\bigg)\!\!\bigg].
\eeqa
We now analyse the first summand on the right-hand side of \eqref{eq:rep.prod.orth}. The decomposition $\X{\al_k}=\XX{\al_k}+\X{\al_k}_{\,0}$, the identity \eqref{eq:prod.n.bin} with $a_r=\XX{\al_r}_-$ and $b_r=\X{\al_r}_{\,0}$, $r=1,\ldots,m+1$, and the definition of $\Rscr_T$ yields 
\[
\bigg(\prod_{\begin{subarray}{c}k=1\\k\neq j_1,\ldots,j_i\end{subarray}}\p {m+1}\!\!\!\!\!\!\!\!\N{\al_k}_-\!\!\bigg)\cdot\X{\al_{j_1},\ldots,\al_{j_i}}_T\in \cl(\Span(\Rscr_T))_{L\p{\,2}(\Pbb)}\,.
\]
By assumption we have that $\xi$ is orthogonal to $\Rscr_T$ and therefore also to $\cl(\Span(\Rscr_T))_{L\p{\,2}(\Pbb)}$. (Note that $\N{\al}_{\,t}$ admits finite moments of every order for all $t\in[0,T]$ and $\al\in \Lm$, cf. \eqref{ass:ex.mom}.) Hence the first summand in \eqref{eq:rep.prod.orth} vanishes and we get
\[
\Ebb\Bigl[\xi\!\!\!\prod_{i=1}\p {m+1}\N{\al_i}_{\,T}\Bigr]=
\sum_{p=0}\p{m-1}\!\sum_{i=p+2}\p{m+1}\sum_{1\leq j_1<\ldots<j_i\leq m+1}\!\!\!\!\!\!\Ebb\bigg[\!\xi\!\!\bigg(\!\!\Bigl(\!\!\!\!\prod_{\begin{subarray}{c}k=1\\k\neq j_1,\ldots,j_i\end{subarray}}\p {m+1}\!\!\!\!\!\!\!\!\N{\al_k}_-\!\!
\prod_{\ell=i-p+1}\p{i}\!\!\!\!\!\!\! \Delta\N{\al_{j_\ell}}\!\Bigr)\cdot
\pb{\N{\al_{j_1},\ldots,\al_{j_{i-p-1}}}}{\N{\al_{j_{i-p}}}}_{\,T}\!\!\bigg)\!\!\bigg].
\]
The processes $\N{\al_{j_{i-p}}}$ and $\N{\al_{j_1},\ldots,\al_{j_{i-p-1}}}$ belong to $\Xscr$ for every $i$ and $p$, because $\Xscr$ is compensated-covariation stable. By assumption, the processes $\pb{\N{\al_{j_1},\ldots,\al_{j_{i-p-1}}}}{\N{\al_{j_{i-p}}}}$ are deterministic. Lemma \ref{lem:int.prod.Z} implies that
$
(\prod_{k\neq j_1,\ldots,j_i}\p{m+1} \N{\al_k}_-
\prod_{\ell=i-p+1}\p{i} \Delta\N{\al_{j_\ell}})\cdot
\pb{\N{\al_{j_1},\ldots,\al_{j_{i-p-1}}}}{\N{\al_{j_{i-p}}}}
$
are processes of integrable variation. Thus we can apply Fubini's theorem and for every summand we get
\begin{equation}\label{eq:orth.N.prod.thm0}
\begin{split}
&\Ebb\bigg[\xi\bigg(\Bigl(\prod_{\begin{subarray}{c}k=1\\k\neq j_1,\ldots,j_i\end{subarray}}\p{m+1} \N{\al_k}_-\prod_{\ell=i-p+1}\p{i} \Delta\N{\al_{j_\ell}}\Bigr)\cdot
\pb{\N{\al_{j_1},\ldots,\al_{j_{i-p-1}}}}
{\N{\al_{j_{i-p}}}}_{\,T}\bigg)\bigg]\\
&=\Ebb\bigg[\xi\Bigl
(\prod_{\begin{subarray}{c}k=1\\k\neq j_1,\ldots,j_i\end{subarray}}\p{m+1}\hspace{-.2cm} \N{\al_k}_-\prod_{\ell=i-p+1}\p{i} \hspace{-.2cm}\Delta\N{\al_{j_\ell}}\Bigr)_{\cdot}\bigg]\cdot
\pb{\N{\al_{j_1},\ldots,\al_{j_{i-p-1}}}}{\N{\al_{j_{i-p}}}}_{\,T}.
\end{split}
\end{equation}
We consider the generic element
\[
K:=\prod_{\begin{subarray}{c}k=1\\k\neq j_1,\ldots,j_i\end{subarray}}\p{m+1}\N{\al_k}_-\prod_{\ell=i-p+1}\p{i}\Delta\N{\al_{j_\ell}}.
\]
After expanding the product we observe that $K_{\,t}$ is equal to a finite sum of terms of type
\[
\prod_{k=1}\p q\N{\al_{i_k}}_{\,t-}\prod_{j=q+1}\p {q+r}\N{\al_{i_j}}_{\,t}, \quad 1\le i_1, \ldots, i_{q+r}\leq m+1 \ \ \mbox{pairwise different}, \quad q+r\le m-1\,.
\]
From the induction hypothesis, \eqref{eq:orth.mon.in.hyp} and Lemma \ref{lem:orth.lef.lim.prod} we now obtain $\Ebb[\xi \,K_{\,t}]=0$ for all $t\in[0,T]$.
Thus every summand \eqref{eq:orth.N.prod.thm0} vanishes, therefore
$
\Ebb[\xi\prod_{i=1}\p {m+1}\N{\al_i}_{\,T}]=0
$ and the proof of the induction step is finished. Consequently, $\xi$ is orthogonal to $\prod_{i=1}\p {m}\N{\al_i}_{\,T}$ for every $m\geq0$  and the proof is complete.
\end{proof}

Now, from Proposition \ref{prop:tot.term.rv.int} and Theorem \ref{MItheorem} we can deduce the main result of this paper.

\begin{theorem}\label{thm:crp.ccs.fam}
Let $\Xscr:=\{\X{\al},\ \al\in\Lm \}\subseteq\Hscr\p2(\Fbb\p\Xscr)$ be a compensated-covariation stable family such that $\aPP{\X{\al}}{\X{\bt}}$ is deterministic for all $\al,\bt\in\Lm$. Suppose moreover that Assumption \ref{ass:tot.pol} is satisfied. Then $\Xscr$ possesses the CRP on $L\p2(\Om,\Fscr\p\Xscr_T,\Pbb)$.
\end{theorem}
\begin{proof}
We set $\Sscr:=\Rbb_+$ and introduce the family $\Zscr:=\Xscr\p\Sscr$ as in \eqref{def:stop.com.cov.fam}. Because of Lemma \ref{lem:def:stop.com.cov.fam}, $\Zscr$ is a compensated-covariation stable family of $\Hscr\p{\,2}$. Starting from $\Zscr$ we define the family $\Rscr_{T}$ as in \eqref{def:term.rv.int}. Clearly $\Zscr$ satisfies all the assumptions of Proposition \ref{prop:tot.term.rv.int} and therefore $\Rscr_{T}$ is total in $L\p2(\Om,\Fscr_T\p\Zscr,\Pbb)$. On the other side, Corollary \ref{CMI}  yields that the family $\Rscr_{T}$ is contained in the closed linear space $\Jscr_T\p\Zscr$ of the terminal variables of the iterated integrals generated by $\Zscr$ and hence $\Jscr_T\p\Zscr=L\p2(\Om,\Fscr_T\p\Zscr,\Pbb)$. Furthermore, the identity $\Fbb\p\Zscr=\Fbb\p\Xscr$ holds. Therefore $\Zscr$ possesses the CRP on $L\p2(\Om,\Fscr_T\p\Xscr,\Pbb)$.
To show that $\Xscr$ possesses the CRP on $L\p2(\Om,\Fscr\p\Xscr_T,\Pbb)$ we only need to show that the space $\Jscr\p\Xscr$ of the iterated integrals generated by $\Xscr$ contains the space $\Jscr\p{\Zscr}$ of those generated by $\Zscr$. But this is obvious because for every $X\in\Zscr$ there exist $\al\in\Lm$ and $u\in\Rbb_+$ such that $X={\X{\al}}\p u$, where the superscript $u$ denotes the operation of stopping at the deterministic time $u$. Clearly, the identities $X={(\X{\al}-\X{\al}_{\,0})}^u+ \X{\al}_{\,0}=J^{(\al)}_1(1\otimes1_{[0,u]})+\X{\al}_{\,0}$ hold. This means that $\Zscr\subseteq\Jscr\p\Xscr$ and hence $\Jscr\p{\Zscr}\subseteq\Jscr\p\Xscr$ implying the CRP for $\Xscr$.
\end{proof}

Applications of Theorem \ref{thm:crp.ccs.fam} to L\'evy processes will be given in Section \ref{sec:prp.lev} below. Further applications and examples will be provided in the concluding Section \ref{sec:appl}.

\section{The CRP for L\'evy Processes}\label{sec:prp.lev}
In this section, given a L\'evy Process $L$ on a fixed time horizon $[0,T]$, $T>0$, we construct families of martingales possessing the CRP on $L\p2(\Om,\Fscr\p L_T,\Pbb)$. We start with a short introduction to L\'evy processes and Poisson random measures.

A \cadlag process $L$ on a probability space $(\Om,\Fscr,\Pbb)$ such that $L_0=0$ is called a L\'evy process if it is stochastically continuous and has homogeneous and independent increments.
Let $L$ be a L\'evy process. By $\Fbb\p L=(\Fscr\p L_{\,t})_{t\in[0,T]}$ we denote the natural filtration of $L$, i.e., the smallest filtration satisfying the usual conditions such that $L$ is adapted. From now on, we restrict ourselves to the probability space $(\Om,\Fscr\p L_T,\Pbb)$ and the filtration $\Fbb\p L$. Because $L_0=0$, $\Fscr\p L_0$ is trivial.

On $(E,\Bscr(E)):=([0,T]\times\Rbb,\Bscr([0,T])\otimes\Bscr(\Rbb))$, where $\Bscr(\cdot)$ denotes the Borel $\sig$-algebra, we introduce the random measure $\M$ (cf.\ \citet{JS00}, Definition II.1.3) by
\[
\M(\om,A):=\sum_{s\geq0}1_{\{\Delta L_{\,s}(\om)\neq0\}}1_{A}(s,\Delta L_{\,s}(\om)),\quad\om\in\Om,\quad A\in\Bscr(E)\,.
\]
We call $\M$ the jump measure of $L$. It is known that $\M$ is a homogeneous Poisson random measure relative to the filtration $\Fbb\p L$, i.e., an integer-valued random measure (cf.\ \citet{JS00}, Definition II.1.13)  such that (i) $\Ebb[\M(A)]=(\lm_+\otimes\nu)(A)$ for every $A\in\Bscr(E)$, where $\lm_+$ is the Lebesgue measure on $[0,T]$ and $\nu$ is a $\sig$-additive measure on $\Rbb$; (ii) for all $s\geq0$ and $A\in\Bscr(E)$ such that $A\subseteq(s,T]\times\Rbb$ the random variable $\M(A)$ is independent of $\Fscr_{\,s}\p L$ (cf.\ \citet{JS00}, Definition II.1.20). The $\sig$-finite measure $\nu$ is the L\'evy measure of $L$, which  satisfies $\nu(\{0\})=0$ and $x\p{\,2}\wedge1\in L\p1(\nu)$. We put $\m:=\lm_+\otimes\nu$.
Now we introduce the compensated Poisson random measure associated with the jump measure of $L$. The system defined by $\Escr:=\{A\in\Bscr(E): \m(A)<+\infty\}$ is a ring of Borel subsets of $E$. For every $A\in\Escr$ we define $\Ms(A):=\M(A)-\m(A)$. The family $\Ms:=\{\Ms(A),\ A\in\Escr\}$ is an \emph{elementary orthogonal random measure} (cf.\ \citet{GS74}, IV, \S\ 4), i.e., for every $A,B\in\Escr$, $\Ms(A)\in L\p{\,2}(\Pbb)$; $\Ebb\bigl[\Ms(A)\Ms(B)\bigr]=\m(A\cap B)$ and if, moreover, $A\cap B=\emptyset$, then it follows  $\Ms(A\cup B)=\Ms(A)+\Ms(B)$. We call $\Ms$ the \emph{compensated Poisson random measure} (associated with $\M$).

Next we briefly recall the stochastic integral with respect to the jump measure $\M$ and the associated compensated Poisson random measure $\Ms$ for measurable functions $f$ on $(E,\Bscr(E))$. First we set \[\m(f):=\int_E f(t,x)\,\m(\rmd\, t,\rmd\, x)\] if the integral on the right-hand side exists. If $f\geq0$, $\m(f)$ always exists. Analogously, we can define the integral of $f$ with respect to $\M$ $\om$-\emph{wise}. If $\int_E |f(t,x)|\,\M(\rmd\, t,\rmd\, x)<+\infty$ a.s., we put $\M(f):=\int_E f(t,x)\,\M(\rmd \,t,\rmd\, x)$ a.s.\
and call $\M(f)$ the \emph{stochastic integral} of $f$ with respect to $\M$. From \citet{K02}, Lemma 12.13, we know that $\M(f)$ exists and is finite a.s.\ if and only if $\m(|f|\wedge1)<\infty$. For $f\in L\p1(\m)$ we have $\Ebb[\M(f)]=\m(f)$.

The stochastic integral with respect to $\Ms$ for deterministic functions in $L\p{\,2}(\m)$ is defined as in \citet{GS74}, IV, \S 4, for a general elementary orthogonal random measure, and we do not repeat the definition in detail. We only recall that in a first step the stochastic integral with respect to $\Ms$ is defined for simple functions in $L\p{\,2}(\m)$ and is then extended to arbitrary functions in $L\p{\,2}(\m)$ by isometry using the denseness of the simple functions: There exists a unique isometric mapping on $L\p{\,2}(\m)$ into $L\p{\,2}(\Pbb)$, again denoted by $\Ms$, such that $\Ms(1_B)=\Ms(B)$, $B\in\Escr$.
If $f\in L\p{\,2}(\m)$, then $\Ms(f)$ is called the stochastic integral of $f$ with respect to the compensated Poisson random measure $\Ms$. The proof of the following proposition is left to the reader.
\begin{proposition}\label{prop:int.M.Ms}
If $f\in L\p1(\m)\cap L\p{\,2}(\m)$, then $\Ms(f)=\M(f)-\m(f)$.
\end{proposition}
The It\^o--L\'evy decomposition of $L$ asserts that there exists a Wiener process relative to $\Fbb$ with \emph{variance function} $\sig\p{\,2}(t):=\sig\p{\,2}\,t$, say $\Ws$, such that the following decomposition holds:
\begin{equation}\label{eq:il.dec}
L_{\,t}=\bt t+\Ws_{\,t}+\M(1_{[0,\,t]\times\{|x|>1\}}\,x)+\Ms(1_{[0,\,t]\times\{|x|\leq1\}}\,x),\quad t\in[0,T],\quad\textnormal{a.s.},
\end{equation}
where $\bt\in\Rbb$, cf.\ \citet{K02}, Theorem 15.4. We call the triplet $(\bt,\sig\p{\,2},\nu)$ the characteristics of $L$ and the process $\Ws$ the Gaussian part of $L$. For a L\'evy process $L$ with Gaussian part $\Ws$ and L\'evy measure $\nu$, we introduce the measure $\mu$ by
\begin{equation}\label{def:mu.meas}
\mu:=\sig\p2\delta_0+\nu\,,
\end{equation}
where $\delta_0$ denotes the Dirac measure in the origin. Since $\nu(\{0\})=0$, without loss of generality we can assume that $f(0)=0$ for every $f\in L\p p(\nu)$, $p\in[1,+\infty]$. With this convention, $L\p p(\nu)$ is a subspace of $L\p p(\mu)$. For any $f\in L\p{\,2}(\mu)$ we introduce the martingale $\X{f}=(\X{f}_{\,t})_{t\geq0}$ by
\begin{equation}\label{eq:en.mar}
\X{f}_{\,t}=f(0)\Ws_t+\Ms(1_{[0,\,t]}1_{\Rbb\setminus\{0\}}f),\quad t\in[0,T].
\end{equation}
We stress that for $f=1_{\{0\}}$ we get $\X{f}=\Ws$ as a special case.

For a measure $\varrho$, we use the notation $\varrho(f):=\int_\Rbb f(x)\,\varrho(\rmd x)$ if the integral on the right exists.
\begin{theorem}\label{thm:en.mar}
For every $f\in L\p{\,2}(\mu)$ the process $\X{f}$ defined by \eqref{eq:en.mar} has the following properties:

\textnormal{(i)} $(\X{f},\Fbb\p L)$ is a L\'evy process and $\Ebb[(\X{f}_{\,t})\p{\,2}]=t\,\mu(f\p{\,2})<+\infty$.

\textnormal{(ii)} $\X{f}\in\Hscr\p{\,2}(\Fbb\p L)$ and $\pb{\X{f}}{\X{g}}_{\,t}=t\mu(fg)$, $f,g\in L\p{\,2}(\mu)\ $.

\textnormal{(iii)} $\Delta\X{f}=f(\Delta L)1_{\{\Delta L\neq0\}}$ a.s.\ and $\X{f}$ is locally bounded if $f$ is bounded.

\textnormal{(iv)} $\X{f}=0$ a.s.\ if and only if $f=0$ $\mu$-a.e.

\textnormal{(v)} $\X{f}$ and $\X{g}$ are orthogonal if and only if $f,g\in L\p{\,2}(\mu)$ are orthogonal.
\end{theorem}

\paragraph{Compensated-Covariation Stable Families.}
As a preliminary step, given a L\'evy process $L$ with characteristics $(\bt,\sig\p{\,2},\nu)$, our aim is to construct compen\-sated-covariation stable families of $\Fbb\p L$-martingales possessing the CRP on $L\p2(\Om,\Fscr_T\p L,\Pbb)$.

Let $(L,\Fbb\p L)$ be a L\'evy process with Gaussian part $\Ws$ and jump measure $\M$; $\Fbb\p{\Ws}=(\Fscr_{\,t}\p{\Ws})_{t\in[0,T]}$ denotes the completion in $\Fscr\p L_T$ of the filtration generated by $\Ws$.  For every $t\in[0,T]$ we introduce the $\sig$-algebra $\Fscr\p\M_{\,t}:=\sig(\{\M(A),\ A\in\Bscr(E),\ A\subseteq[0,t]\times\Rbb\})\vee\Nscr(\Pbb)$, where $\Nscr(\Pbb)$ denotes the system of the $\Pbb$-null sets of $\Fscr\p{\,L}_{\,T}$, and put $\Fbb\p\M=(\Fscr_{\,t}\p{\M})_{t\in[0,T]}$. It is easy to see that $\Fbb\p L=\Fbb\p{\Ws}\vee\Fbb\p\M$.

We shall consider systems $\Cscr$ of real functions with the following properties:
\begin{assumption}\label{ass:good.sys}
(i) $\Cscr\subseteq L\p{\,1}(\mu)\cap L\p{\,2}(\mu)$;
(ii) $\Cscr$ is total in $L\p{\,2}(\mu)$;
(iii) $\Cscr$ is stable under multiplication and $1_{\Rbb\setminus\{0\}}f\in\Cscr$ whenever $f\in\Cscr$;
(iv) $\Cscr$ is a system of bounded functions.
\end{assumption}
Notice that a system $\Cscr$ satisfying Assumption \ref{ass:good.sys} always exists: An example can easily be constructed taking $\Cscr:=\{f=c1_{\{0\}}+1_{(a,\,b]},\ a,b\in\Rbb:a<b,\ 0\notin[a,b];\ c\in\Rbb\}\cup\{0\}$.

For a system $\Cscr$ satisfying Assumption \ref{ass:good.sys} we introduce the set $\tilde\Cscr$ of all $\tilde f:=1_{\Rbb\setminus\{0\}}\,f$, $f\in\Cscr$. Recalling the convention above, we observe that $\tilde\Cscr\subseteq L\p{\,1}(\nu)\cap L\p{\,2}(\nu)$, is total in $L\p{\,2}(\nu)$ and is stable under multiplication. We also define the family
\begin{equation}\label{eq:good.fam}
\Xscr_\Cscr:=\{\X{f},\quad f\in\Cscr\}
\end{equation}
where the martingales $\X{f}$ are introduced in \eqref{eq:en.mar}.
Then the following proposition holds:
\begin{proposition}\label{prop:good.fam.com.cov}
The family $\Xscr_\Cscr$ is a compensated-covariation stable family of $\Fbb\p L$-martingales in $\Hscr\p{\,2}(\Fbb\p L)$. Moreover $\Xscr_\Cscr$ generates the filtration $\Fbb\p L$; $\Ebb[\exp(\lm|X_{\,t}|)]<+\infty$ for every $X\in\Xscr_\Cscr$, $\lm>0$, $t\in[0,T]$, and $\pb XY$ is deterministic for every $X,Y\in\Xscr_\Cscr$.
\end{proposition}
\begin{proof}
It is clear that $\Xscr_\Cscr\subseteq\Hscr\p{\,2}(\Fbb\p L)$. Now we show that $\Fscr\p{\,\Xscr_\Cscr}_{\,T}=\Fscr\p L_T$ ($=\Fscr\p{\,\Ws}_{\,T}\vee\Fscr\p{\,\M}_{\,T}$). By assumption we have $f\in L\p1(\mu)\cap L\p2(\mu)$ for $f\in\Cscr$. From this it follows that $\tilde f:=1_{\Rbb\setminus\{0\}}\,f$ belongs to $L\p1(\nu)\cap L\p2(\nu)$ and an application of Proposition \ref{prop:int.M.Ms} yields $\X{f}_{\,t}=f(0)\,\Ws_t+\M(1_{[0,\,t]}\tilde f)-t\nu(\tilde f)$, $t\geq 0$. We set \[\Gscr:=\sigma(\{\M(1_{[0,\,t]}\tilde f),\quad t\in[0,T]\ , f\in\tilde\Cscr\})\vee\Nscr(\Pbb).\] Recall that $\tilde\Cscr$ is total in $L\p2(\nu)$.
It is evident that $\Fscr_T\p{\Xscr_\Cscr}=\Fscr_T\p{\Ws}\vee\Gscr$ and therefore it is enough to verify that $\Gscr=\Fscr\p{\,\M}_{\,T}$. Recalling the definition of $\Fscr\p{\,\M}_{\,T}$ we easily obtain that $\M(1_{[0,\,\cdot]}g)$ is $\Fscr\p{\,\M}_{\,T}$-measurable for every nonnegative measurable function $g$. Since $\tilde\Cscr\subseteq L\p{\,1}(\nu)$, we can write $\M(1_{[0,\,t]}\tilde f)=\M(1_{[0,\,t]}\tilde f^+)-\M(1_{[0,\,t]}\tilde f^-)$  a.s.\ for every $\tilde f\in\tilde\Cscr$ and $t\geq 0$, which is $\Fscr\p{\,\M}_{\,T}$-measurable. This yields the inclusion $\Gscr\vee\nolinebreak[4]\Nscr(\Pbb)\subseteq\Fscr\p{\,\M}_{\,T}$. Conversely, let $B_n\subseteq\Rbb$ be such that $B_n\uparrow\Rbb$ and $\nu(B_n)<+\infty$ for all $n\geq 1$. Using the monotone class theorem we deduce that $\M(1_{[0,\,t]\times B_n}\,g)$ is $\Gscr\vee\Nscr(\Pbb)$-measurable for every bounded measurable function $g$ on $E=[0,T]\times \Rbb$ and hence for $g=1_A$ with $A\in \Bscr(E)$. Finally, letting $n
 \rightarrow\infty$, we observe that $\M(([0,t]\times\Rbb)\cap A)$ is $\Gscr\vee\Nscr(\Pbb)$-measurable for all $A\in \Bscr(E)$, proving the inclusion $\Fscr\p{\,\M}_{\,T}\subseteq\Gscr\vee\Nscr(\Pbb)$. Next we show that the family $\Kscr$ of monomials generated by $\Xscr_\Cscr$ is total in $L\p2(\Om,\Fscr\p{\Xscr_\Cscr}_T,\Pbb)$ and hence, from the previous step, in $L\p2(\Om,\Fscr\p L_T,\Pbb)$. Indeed,
the L\'evy measure of $\X{f}$ is $\nu\p{\tilde f}$, where $\nu\p {\tilde f}(\{0\}):=0$ and $\nu\p {\tilde f}(B):=\nu\circ \tilde f\p{\,-1}(B)$, $B\in\Bscr(\Rbb\setminus\{0\})$. Because each function $\tilde f$ in $\tilde \Cscr$ is bounded, $\nu\p{\tilde f}$ has bounded support. From \citet{S99}, Lemma 25.6 and 25.7, we can deduce that $\X{f}_{\,t}$ has finite exponential moments of every order for all $t\in[0,T]$ and $f\in\Cscr$. Now the claim follows from Theorem \ref{thm:den.pol}. From Theorem \ref{thm:en.mar} it is clear that for all $f,g\in\Cscr$ the brackets $\aPP{\X{f}}{\X{g}}$ are deterministic.
It remains to show that $\Xscr_\Cscr$ is compensated-covariation stable. Let $f,g\in\Cscr$ and define $h:=fg1_{\Rbb\setminus\{0\}}$. We notice that $h$ again belongs to $\Cscr$. Using \eqref{def:compcov}, \eqref{def:cov}, Theorem \ref{thm:en.mar} (ii), (iii) and  Proposition \ref{prop:int.M.Ms}, we can compute
\begin{eqnarray*}
\X{f,g}_{\,t}
&:=&
[\X{f},\X{g}]_{\,t}-\pb{\X{f}}{\X{g}}_{\,t}
\\&=&f(0)g(0)\sigma^2\,t+
\sum_{0\leq s\leq t}\tilde f(\Delta L_{\,s})\tilde g(\Delta L_{\,s})1_{\{\Delta L_{\,s}\neq0\}}-\mu(fg)\,t\\
&=&\sum_{0\leq s\leq t}h(\Delta L_{\,s})1_{\{\Delta L_{\,s}\neq0\}}-\nu(h)\,t=
\Ms(1_{[0,t]}h)=\X{h}_{\,t},\quad t\in[0,T],\textnormal{ a.s.}
\end{eqnarray*}
Hence $\X{f,g}$ belongs to $\Xscr_\Cscr$ proving that $\Xscr_\Cscr$ is a compensated-covariation stable family.
\end{proof}
As a consequence of Proposition \ref{prop:good.fam.com.cov} and of Theorem \ref{thm:crp.ccs.fam} we get the following result:

\begin{proposition}\label{prop:com.cov.lev.prp}
Let $(L,\Fbb\p L)$ be a L\'evy process with characteristics $(\bt,\sig\p{\,2},\nu)$, $\mu$ be the measure defined in \eqref{def:mu.meas} and $\Cscr\subseteq L\p2(\mu)$ satisfy Assumption \ref{ass:good.sys}. Then the family $\Xscr_\Cscr$ defined in \eqref{eq:good.fam} possesses the CRP on $L\p2(\Om,\Fscr\p L_T,\Pbb)$.
\end{proposition}
\begin{proof}
From Proposition \ref{prop:good.fam.com.cov} and Theorem \ref{thm:den.pol} it follows that the family $\Xscr_\Cscr$ satisfies all the assumptions of Theorem \ref{thm:crp.ccs.fam} and therefore it possesses the CRP on $L\p2(\Om,\Fscr\p L_T,\Pbb)$.
\end{proof}
\paragraph{General Families of Martingales with the CRP}
Let $(L,\Fbb\p L)$ be a L\'evy process with characteristic triplet $(\bt,\sig\p2,\nu)$ and let $\mu$ be as in \eqref{def:mu.meas}.
With a system $\Tscr\subseteq L\p{\,2}(\mu)$, we associate the family $\Xscr_\Tscr$ by
\begin{equation}\label{def:fam.tscr}
\Xscr_\Tscr:=\{\X{f},\quad f\in\Tscr\}\,.
\end{equation}
Now we give necessary and sufficient conditions on $\Tscr$ for $\Xscr_\Tscr$ to possess the CRP on $L\p2(\Om,\Fscr\p L_T,\Pbb)$. 

We stress that in general the family $\Xscr_\Tscr$  need not satisfy all the assumptions of Theorem \ref{thm:crp.ccs.fam}. In particular, the family $\Xscr_\Tscr$ need not be compensated-covariation stable or possess exponential moments.
\begin{theorem}\label{thm:tot.sys.h2.prp}
Let $\Tscr$ be a system of functions in $L\p{\,2}(\mu)$, where $\mu$ is defined in \eqref{def:mu.meas}. The family $\Xscr_\Tscr$ defined in \eqref{def:fam.tscr} possesses the CRP with respect to $\Fbb\p L$ if and only if $\Tscr$ is total in $L\p{\,2}(\mu)$.
\end{theorem}
\begin{proof}
First we assume that the family $\Xscr_\Tscr$ possesses the CRP and show that $\Tscr$ is total in $L\p{\,2}(\mu)$. We choose a function $h\in L\p2(\mu)$ which is orthogonal to $\Tscr$ in $L\p2(\mu)$. By Theorem \ref{thm:en.mar} (v), the martingale $\X{h}\in\Hscr\p{\,2}$ associated with $h$ is orthogonal to $\Xscr_\Tscr$. For an elementary iterated integral
\begin{equation*}
J_{n}\p{(f_1,\ldots, f_{n})}(F)_{\,t}
:=\int_0^t J_{n-1}^{(f_1,\ldots, f_{n-1})}(F_0\otimes\cdots\otimes F_{n-1})_{u-}\,F_{n}(u)\, \rmd {\X{f_{n}}_u},\ \ t\in [0,T], \ \ n\geq 1\,,
\end{equation*}
with respect to $(\X{f_1}, \X{f_2}, \ldots\ , \X{f_n})$, $f_k\in\Tscr$, $k=1,\ldots,n$, where $F=F_0\otimes\cdots\otimes F_{n}$ is an elementary function (see Definition \ref{elem.it.int.}), we obtain
$$
\aPP{J_{n}\p{(f_1,\ldots, f_{n})}(F)}{\X{h}}_t=\int_0^t J_{n-1}^{(f_1,\ldots, f_{n-1})}(F_0\otimes\cdots\otimes F_{n-1})_{u-}\,F_{n}(u)\, \rmd \aPP{{\X{f_{n}}}}{\X{h}}_u=0,\ \ t\in [0,T]\,.
$$
Hence the martingales $\X{h}$ and $J_{n}\p{(f_1,\ldots, f_{n})}(F)$ are orthogonal. It is clear that $\X{h}$ is also orthogonal to $\Jscr_0$. This implies that the terminal value $\X{h}_T$ is orthogonal to $J_{n}\p{(f_1,\ldots, f_{n})}(F)_T$ and also orthogonal to $\Jscr_{0,T}$ in $L\p2(\Om,\Fscr_T\p L,\Pbb)$. Recalling the construction of the space $\Jscr\p{\Xscr_\Tscr}_T$ of the iterated integrals generated by $\Xscr_\Tscr$ (cf. Definition \ref{def:sp.mul.int} and Proposition \ref{prop:equivCRP} (i)), we observe that the system of elementary iterated integrals of order $n$ ($n\geq0$) is total in $\Jscr\p{\Xscr_\Tscr}_T$. Consequently, $\X{h}_T$ is orthogonal to $\Jscr\p{\Xscr_\Tscr}_T$. By definition of the CRP, $\Jscr\p{\Xscr_\Tscr}_T=L\p2(\Om,\Fscr_T\p L,\Pbb)$ and hence $\X{h}_{\,T}=0$. From this we deduce that the martingale $\X{h}$ is indistinguishable from the null-process and by Theorem \ref{thm:en.mar} (iv) it follows $h=0$ $\mu$-a.e. This proves that $\Tscr$ is total in $L
 \p{\,2}(\mu)$. Conversely, we now assume that $\Tscr$ is total in $L\p{\,2}(\mu)$ and show that $\Xscr_\Tscr$ has the CRP on $L\p2(\Om,\Fscr_T\p L,\Pbb)$. For this purpose we consider a system $\Cscr$ satisfying Assumption \ref{ass:good.sys}. From Proposition \ref{prop:com.cov.lev.prp} we know that $\Xscr_\Cscr$ has the CRP on $L\p2(\Om,\Fscr_T\p L,\Pbb)$. We denote by $\Jscr\p{\Xscr_\Cscr}$ the space of iterated integrals generated by $\Xscr_\Cscr$. It is enough to prove that $\Jscr\p{\Xscr_\Cscr}=\Jscr\p{\Xscr_\Tscr}$. Because $\Tscr$ is total in $L\p{\,2}(\mu)$ and the mapping $f\mapsto\X{f}$ is linear and isometric from $L\p{\,2}(T\mu)$ into $\Hscr^2$ (see Theorem \ref{thm:en.mar}), we immediately establish $\cl(\Span(\Xscr_\Tscr))_{\Hscr^2}=\Xscr_{L\p{\,2}(\mu)}$ and hence the inclusion $\Xscr_\Cscr\subseteq\cl(\Span(\Xscr_\Tscr))_{\Hscr^2}$ holds. Using Theorem \ref{thm:cl.st.it} for $\Zscr=\Xscr_\Tscr$ we conclude $\Jscr\p{\Xscr_\Cscr
 }=\Jscr\p{\Xscr_\Tscr}$, proving the claim. This completes the proof of the theorem.
\end{proof}
We remark that if $L$ is a square integrable L\'evy process, or equivalently the function $x$ belongs to $L\p2(\nu)$, we can choose the total system $\Tscr\subseteq L\p2(\mu)$ in such a way that the function $h:=1_{\{0\}}+x$ belongs to $\Tscr$. In this case we have that $\X{h}=\Ls$, where $\Ls_{\,t}:=L_{\,t}-\Ebb[L_{\,t}]$, $t\in[0,T]$. In other words, the L\'evy process $\Ls$ can be included in the family $\Xscr_\Tscr$.

An important question is in which cases it is possible to choose a family $\Xscr_\Tscr$ consisting of finitely many martingales and possessing the CRP. The next corollary explains that this is possible only in a rather few cases.
\begin{corollary}\label{cor:fin.fam}
Let $(L,\Fbb\p L)$ be a L\'evy process with characteristics $(\bt,\sig\p{\,2},\nu)$. The following statements are equivalent:
\textnormal{(i)} There exists a finite family $\Xscr_\Tscr$  possessing the CRP on $L\p2(\Om,\Fscr_T\p L,\Pbb)$;
\textnormal{(ii)} $L\p{\,2}(\mu)$ is finite-dimensional;
\textnormal{(iii)} $\nu$ has finite support.
\end{corollary}

The situation discussed in Corollary \ref{cor:fin.fam} occurs if $L$ is a \emph{simple} L\'evy process, i.e., it is of the form
\[
L_{\,t}:=\Ws_{\,t}+\alpha_1\,N_{\,t}\p1 +\ldots+\alpha_m\,N_{\,t}\p m
\,\quad t\geq0\,,\]
where $(\Ws,\Fbb\p L) $ is a Brownian motion with variance function $\langle\Ws,\Ws\rangle_{\,t}=\sig\p{\,2}\,t$; $(N\p j,\Fbb\p L)$ a homogeneous Poisson process with parameter $\gamma_j>0$, $j=1,\ldots,m$, and $(N\p1,\ldots,N\p m)$ is an independent vector of processes; $\alpha_1,\ldots,\alpha_m$ are real numbers.

If $\Tscr\subseteq L\p{\,2}(\mu)$ is a \emph{complete orthogonal system}, say $\Tscr:=\{f_n,\ n\geq1\}$ (note that $L\p{\,2}(\mu)$ is a separable Hilbert space), then the associated family $\Xscr_\Tscr$ consists of countably many mutually orthogonal martingales (cf.\ Theorem \ref{thm:en.mar} (v)) and Theorem \ref{thm:tot.sys.h2.prp}, Theorem \ref{thm:mul.int.orth.sum} and \eqref{eq:crp.or.sys} yield the following theorem. Note that $\Jscr_0=\Jscr_{0,T}=\mb R$ because the $\sig$-field $\Fscr^{\Xscr_\Tscr}_0=\Fscr\p{\,L}_0$ is $\mb P$-trivial.
\begin{theorem}\label{thm:cou.ort.sys.prp}
Let $(L,\Fbb\p L)$ be a L\'evy process with characteristics $(\bt,\sig\p2,\nu)$ and $\Tscr=\{f_n,\ n\geq1\}$ be a complete orthogonal system in $L\p{\,2}(\mu)$ where $\mu$ is as in \eqref{def:mu.meas}. Then the associated family $\Xscr=\Xscr_\Tscr$ has the CRP on $L\p2(\Om,\Fscr_T\p L,\Pbb)$ and the following decompositions hold:
\begin{equation}\label{eq:crp.or.sys.lp}
\Hscr^2(\mb F^{\Xscr})=\mb R\oplus\bigoplus_{n=1}\p\infty
\bigoplus_{(j_1,\ldots,j_n)\in\Nbb\p n} \Jscr_{n}\p{(f_{j_1},\ldots,f_{j_n})}, \ \ L\p2(\Om,\Fscr\p{\Xscr}_T,\Pbb)=
\Rbb\oplus\Bigg(\bigoplus_{n=1}\p\infty
\bigoplus_{(j_1,\ldots,j_n)\in\Nbb\p n}\Jscr_{n,T}\p{(f_{j_1},\ldots,f_{j_n})}\Bigg)\,,
\end{equation}
where, for $f_1,\ldots,f_n\in\Tscr$, $\Jscr_{n}\p{(f_{1},\ldots,f_{n})}$ denotes the linear space of $n$-fold iterated integrals with respect to $(\X{f_1},\ldots,\X{f_n})$; $n\geq1$.
\end{theorem}
\paragraph{Multiple It\^o Integrals and Iterated Integrals} \label{sec:itoiter}
The aim of this subsection is to establish the relation between \emph{multiple} It\^o integrals introduced in \citet{I56} and \emph{iterated} integrals introduced in Section \ref{sec:def.it.int}.

Let $L$ be a L\'evy process with characteristics $(\beta,\sigma^2,\nu)$. The measure $\mu$ is defined in \eqref{def:mu.meas} while the martingale $\X{f}$ for $f\in L\p2(\mu)$ in \eqref{eq:en.mar} above.

We introduce the measures $\zt$ and $\et$ by
$$
\zt(B):=\int_B x\p{\,2} \nu(\d x), \ B\in \Bscr(\mb R),\quad \et=\sig\p{\,2}\delta_{\,0}+\zt\,.
$$
For any square integrable function $G$ on $([0,T]\times \mb R)^n$ with respect to the measure $(\lambda_+\otimes \et)^n$ (recall that $\lm_+$ denotes the Lebesgue measure on $\Rbb_+$) the \textit{multiple} It\^o integral $I_n(G)$ of order $n$ is defined as follows: For $n=1$ and $G_1\in L^2(\lm_+\otimes\et)$ we put
$$
I_1(G_1)_t:=G_1(\cdot,0)\cdot \Ws_t+\Ms(1_{[0,t]}\,x\,G_1),\quad t\in [0,T]\,.
$$
Now, if $G=1_{A_1}\otimes\cdots\otimes1_{A_n}$ with $A_1, \ldots A_n\in \Bscr([0,T]\times \mb R)$ pairwise disjoint and $(\lambda_+\otimes \et)(A_i)<+\infty$,  $i=1,\ldots, n$, then we define
$$
I_n(G)_t:=I_1(1_{A_1})_t\cdots I_1(1_{A_n})_t,\quad t\in [0,T]\,.
$$
This system of functions $G$ is total in the $L^2$-space over $([0,T]\times \mb R)^n$ equipped with the product measure $(\lambda_+\otimes \et)^n$ (cf. \citet{I51}, Theorem 2.1). The mapping $G\mapsto I_n(G)_t$ can now be extended to the space $L^2(([0,T]\times \mb R)^n, (\lambda_+\otimes \et)^n)$ by linearity and continuity (see \citet{I56} for more details).

Let $f_1,\ldots , f_n$ from $L^2(\mu)$ be normalized. Note that then the measures $m^{(f_i)}$ on $[0,T]$ associated with $\aP {\X{f_i}}$ are equal to $\lm_+$, $i=1,\ldots, n$, and therefore the space $L^2(\Om\times [0,T]^n,m^{(f_1,\ldots, f_n)}_{\mb P})$ coincides with $L^2(\Om\times[0,T]^n,\mb P\otimes \lm_+^n)$. For any $F$ from $L^2([0,T]^n):=L^2([0,T]^n,\lm_+^n)$ the function $1\otimes F$ belongs to the Hilbert space $L^2(\Om\times[0,T]^n,\mb P\otimes \lm_+^n)$ and we can define $J_n^{(f_1,\ldots f_n)}(1\otimes F)=(J_n^{(f_1,\ldots f_n)}(1\otimes F)_t)_{t\in[0,T]}$ with respect to the martingales $(\X{f_1},\ldots, \X{f_n})$ (cf.\ Definition \ref{def:mul.st.int}).

With a real function $g$ defined on $\mb R$ we associate the function $\hat{g}$ defining
\begin{equation}\label{def:fx}
\hat{g}(x)=\begin{cases}
xg(x),& \mbox{if} \ x\not=0,\\
g(0), & \mbox{if} \ x=0 \,.
\end{cases}
\end{equation}
If $g\in L^2(\et)$, then the associated function $\hat{g}$ belongs to $L^2(\mu)$ and conversely. Obviously, we have $\|g\|_{L^2(\et)}=\|\hat{g}\|_{L^2(\mu)}$. Now we consider normalized $g_1,\ldots,g_n\in L\p2(\et)$ and put $f_i:=\hat{g}_i$, $i=1, \ldots, n$. Then $f_1, \ldots, f_n\in L^2(\mu)$ are normalized functions, too. Furthermore we choose $F\in L^2([0,T]^n)$. The function $G=F\,g_1\otimes \cdots \otimes g_n$ clearly belongs to $L^2((\lm_+\otimes \et)^n)$ and hence the multiple integral $I_n(G)$ is well-defined. In the next proposition we denote the set of all permutations of $\{1,\ldots n\}$ by $\Pi_n$ and its generic element by $\pi=(i_1,\ldots, i_n)$. The mapping $\sigma_\pi$ on $[0,T]^n$ is defined as $\sigma_\pi(t_1,\ldots,t_n)=(t_{i_1}, \dots, t_{i_n})$.
\begin{proposition}\label{prop:rel.it.iter}
Let $F\in L^2([0,T]^n)$, $g_1,\ldots,g_n\in L\p2(\et)$ be normalized and define $f_i:=\hat{g}_i$, $i=1, \dots, n$, as in \eqref{def:fx}. Then the following relation between the multiple It\^o integral and the iterated integral holds:
\[
I_n(F\,g_1\otimes \cdots \otimes g_n)_t=\sum_{\pi=(i_1,\ldots, i_n)\in \Pi_n} J_n^{(f_{i_1},\ldots,f_{i_n})}(1\otimes (F\circ \sigma_\pi))_t,\quad t\in [0,T]\,.
\]
In particular, recalling the notation  $M_t\p{(n)}$ of \eqref{eq:def.Mn},
\[
I_n(F\,1_{M_t^{(n)}}\,g_1\otimes \cdots \otimes g_n)_t= J_n^{(f_1,\ldots,f_n)}(1\otimes F)_t,\quad t\in [0,T]\,.
\]
\end{proposition}
\begin{proof}
Let $F=F_1\otimes\cdots \otimes F_n$ be such that $F_i=1_{B_i}$, $i=1,\ldots, n$; $B_1, \ldots, B_n\in \Bscr([0,T])$ and pairwise disjoint. The normalized functions $g_1, \ldots, g_n$ are chosen such that $g_i=c_i\,1_{C_i}$, $C_i\in \Bscr(\mb R)$, $0<\et(C_i)<+\infty$, $c_i>0$, $i=1, \ldots, n$. Using the definition of It\^o's multiple integral and of $\X{f_i}$ (cf.\ \eqref{eq:en.mar}) with $f_i:=\hat{g}_i$ we obtain
$$
I_n(F\,g_1\otimes\ldots\otimes g_n)_t=\prod_{i=1}^n \int_0^t F_i(u)\,\d \X{f_i}_u,\quad t\in [0,T]\,.
$$
Because $B_1, \ldots, B_n$ are pairwise disjoint, setting $Z\p{\,i}:=\int_0^\cdot F_i(u)\,\d \X{f_i}_{
\,u}$, for  $i,j=1,\ldots,n$, such that $i\not=j$, it follows $[Z\p{\,i},Z\p{\,j}]=0$. Therefore, an application of It\^o's formula yields
$$
\prod_{i=1}^n \int_0^t F_i(u)\,\d \X{f_i}_u=\sum_{i=1}^n\int_0^t \prod_{\fraco{k=1}{k\not=i}}^n \int_0^{u-} F_k(s)\,\d \X{f_k}_s F_i(u)\, \d \X{f_i}_u,\quad t\in [0,T]\,.
$$
By an induction argument we can now derive
\begin{equation}\label{id}
I_n(F\,g_1\otimes\ldots\otimes g_n)_t=\sum_{\pi=(i_1,\ldots, i_n)\in \Pi_n} J_n^{(f_{i_1},\ldots,f_{i_n})}(F\circ \sigma_\pi)_t,\quad t\in [0,T]\,.
\end{equation}
The set of functions $F$ considered in the previous step is total in $L^2([0,T]^n)$ (cf.\ e.g. \citet{I51}, Theorem 2.1). The right and left hand sides of identity \eqref{id} are linear and continuous in $F$ on $L^2([0,T]^n)$. Hence \eqref{id} is valid for all $F\in L^2([0,T]^n)$.
Now we fix $F\in L^2([0,T]^n)$. From the previous step we know that \eqref{id} is valid for all $g_1, \ldots, g_n$ chosen as in the first step of the proof. Clearly, the set of all normalized $g_1=c_1\,1_{C_1}$ with $C_1\in \Bscr(\mb R)$ and $0<\et(C_1)<+\infty$, $c_1>0$, is total in $L^2(\et)$. The right and left hand sides of identity \eqref{id} being linear and continuous in $g_1$ on $L^2(\et)$, identity \eqref{id} remains valid for all $g_1\in L^2(\et)$. Repeating the argument for all $i=2, \ldots, n$ yields that \eqref{id} is valid for all $g_1, \ldots, g_n\in L^2(\et)$. The proof of the proposition is now complete.
\end{proof}
By $\Hscr_{n,T}$ we denote the closed linear subspace of $L^2(\Pbb):=L^2(\Om,\mc F^L_T,\Pbb)$ consisting of the terminal values $I_n(G)_T$ of all the multiple It\^o integrals of order $n$. It\^o's chaos expansion (see \citet{I56}, Theorem 2) states the following decomposition of $L\p2(\Pbb)$:
\begin{equation}\label{eq:it.ch}
L^2(\Pbb)=\Rbb\oplus\bigoplus_{n=1}\p\infty\Hscr_{n,T}\,.
\end{equation}
The decomposition \eqref{eq:it.ch} can be deduced as a consequence of Theorem \ref{thm:tot.sys.h2.prp}. Indeed, the linear space $\Jscr_T$ of terminal values of all iterated integrals with respect to $\Xscr_\Tscr:=\{\X{f}: \ f\in \Tscr\}$, where $\Tscr\subseteq L\p2(\mu)$ is a total set, is dense in $L^2(\Pbb)$. But from Proposition \ref{prop:rel.it.iter}, we deduce $\Jscr_T\subseteq\Rbb\oplus\bigoplus_{n=1}\p\infty \Hscr_n(T)$.

\smallskip
We conclude this subsection deriving Theorem \ref{thm:tot.sys.h2.prp} from the It\^o chaos decomposition.
If $(g_k)_{k\geq1}$ is a complete orthonormal system in $L^2(\et)$,
$
(g_{k_1\ldots\ k_n})_{(k_1, \dots, k_n)\in \mb N^n}:=(g_{k_1}\otimes \cdots \otimes g_{k_n})_{(k_1, \dots, k_n)\in \mb N^n}
$
is a complete orthonormal system in $L^2(\mb R^n,\et^n)$. For any function $G\in L^2(([0,T]\times \mb R)^n, (\lambda_+\otimes\et)^n)$ and $(t_1, \ldots, t_n)\in [0,T]^n$ we introduce the function $G_{t_1, \ldots, t_n}$ by
$
G_{t_1, \ldots, t_n}(x_1, \ldots, x_n)=G((t_1,x_1), \ldots, (t_n,x_n))
$, where $(x_1, \ldots, x_n)\in \mb R^n$.
Clearly, the set of all $(t_1, \ldots, t_n)\in [0,T]^n$ such that $G_{t_1, \ldots, t_n}$ does not belong to $L^2(\mb R^n,\et^n)$ is measurable and has zero $\lambda_+^n$-measure. Hence, without loss of generality, we can put $G_{t_1, \ldots, t_n}\equiv 0$ for such points. Consequently, $G_{t_1, \ldots, t_n}$ belongs to $L^2(\mb R^n,\et^n)$ for all $(t_1, \ldots, t_n)$ from $[0,T]^n$. Developing $G_{t_1, \ldots, t_n}$ as Fourier series gives
$
G_{t_1, \ldots, t_n}=\sum_{(k_1, \ldots, k_n)\in \mb N^n} c(t_1, \ldots , t_n; k_1, \ldots k_n)\, g_{k_1\ldots k_n}
$
in $L^2(\mb R^n,\et^n)$ where
$
c(t_1, \ldots , t_n; k_1, \ldots k_n)=(G_{t_1, \ldots, t_n},g_{k_1\ldots k_n})_{L^2(\mb R^n,\et^n)}
$
is measurable in $(t_1, \ldots, t_n)\in [0,T]^n$. Now we verify that
\begin{equation}\label{exp}
G=\sum_{(k_1, \ldots, k_n)\in \mb N^n} c(\cdot, k_1, \ldots k_n) \, g_{k_1\ldots k_n}
\end{equation}
(convergence in $L^2(([0,T]\times\mb R)^n,(\lambda_+\otimes\et)^n)$). We already know that the series converges to $G_{t_1, \ldots, t_n}$ in $L^2(\mb R^n,\et^n)$ and hence in $\et^n$-measure for every $(t_1, \ldots, t_n)\in [0,T]^n$. Therefore the series converges to $G$ in $(\lambda_+\otimes \et)^n$-measure. We introduce the notation $\Nbb^n_m:=\Nbb^n\setminus\{1,\cdots, m\}^n$. For proving the claim we show that
$
\sum_{(k_1, \ldots, k_n)\in\Nbb^n_m} c(\cdot, k_1, \ldots k_n)\, g_{k_1\ldots k_n}
$
converges in $L^2(([0,T]\times\mb R)^n,(\lambda_+\otimes\et)^n)$ to zero as $m\rightarrow\infty$. From the orthonormality of $(g_{k_1\ldots k_n})_{(k_1, \dots, k_n)\in \mb N^n}$ we have
\begin{equation*}
\Vert\sum_{(k_1, \ldots, k_n)\in\Nbb^n_m} c(\cdot, k_1, \ldots k_n)\, g_{k_1\ldots k_n}\Vert^2_{L^2((\lambda_+\otimes\et)^n)}
=\int_{[0,T]^n}\sum_{(k_1, \ldots, k_n)\in\Nbb^n_m} c^2(t_1, \ldots, t_n; k_1, \ldots, k_n)\, \d \lambda_+^n(t_1, \ldots, t_n)\,,
\end{equation*}
the right hand side converging to zero as $m\rightarrow\infty$ because
$\sum_{(k_1, \ldots, k_n)\in\Nbb^n_m} c^2(t_1, \ldots , t_n; k_1, \ldots k_n)\rightarrow 0$ as $m\rightarrow\infty$ for every $(t_1, \ldots, t_n)\in [0,T]^n$ (note that $c(t_1, \ldots , t_n; k_1, \ldots k_n)$ are Fourier coefficients) and is bounded by the integrable function $\sum_{(k_1, \ldots, k_n)\in\Nbb^n} c^2(t_1, \ldots , t_n; k_1, \ldots k_n)$:
\begin{equation*}
\int_{[0,T]}\sum_{(k_1, \ldots, k_n)\in\Nbb^n} c^2(t_1, \ldots , t_n; k_1, \ldots k_n)\, \d \lambda_+^n(t_1,\ldots, t_n)=\|G\|^2_{L^2((\lambda_+\otimes\et)^n)}<+\infty\,.
\end{equation*}
Hence the Fourier expansion \eqref{exp} of $G$ is verified.
As a result, using the continuity of $I_n$ and Proposition \ref{prop:rel.it.iter} above, we get
\begin{equation*}
\begin{split}
I_n(G)_t&=\sum_{(k_1, \ldots, k_n)\in\Nbb^n} I_n(c(\cdot, k_1, \ldots, k_n)\, g_{k_1\cdots k_n})_t\\
&{}=\sum_{(k_1, \ldots, k_n)\in\Nbb^n}\sum_{\pi=(i_1, \ldots, i_n)\in \Pi_n} J_n^{(f_{k_{i_1}}, \ldots, f_{k_{i_n}})}(c(\cdot,k_1,\ldots,k_n)\circ \sigma_\pi)_t\,.
\end{split}
\end{equation*}
Applying It\^o's chaos expansion, from this we get $
L^2(\mb P)=\mb R\otimes\bigoplus_{n=1}\p\infty \Jscr_{n,T}$
where $\Jscr_{n,T}$ is the closed linear space of iterated integrals of order $n$ with respect to the \emph{orthogonal} family  $\{\X{f_k}: \ k\in \mb N\}$ of martingales (cf.\ Definition \ref{def:sp.mul.int}) and $f_k:=\hat{g}_k,\ k\in\Nbb$ (cf.\ \eqref{def:fx}). Clearly, $(f_k)_{k\in\Nbb}$ is a complete orthonormal system of $L^2(\mu)$. Conversely, we can also start from a complete orthonormal system $(f_k)_{k\in\Nbb}$ of $L^2(\mu)$ and construct the complete orthonormal system $(g_k)_{k\in\Nbb}$ of $L^2(\et)$ from it. 

\section{Applications}\label{sec:appl}
In this last section we shall provide a few applications of the main theorem of the present paper, Theorem \ref{thm:crp.ccs.fam} above. This will illustrate, in particular, that several important results on the CRP which have been known before are an immediate consequence of Theorem \ref{thm:crp.ccs.fam}. However, there will also be discussed new examples which are beyond the scope of known results. We start with families of continuous local martingales which are pairwise Gausssian and state the result in a general form which to our knowledge has not been established before. We proceed with families of compensated Poisson processes. Then we pass on to concrete applications of Section \ref{sec:prp.lev} on the CRP of families of martingales related with L\'evy processes: Teugels martingales, families of martingales constructed from Hermite polynomials and from Haar systems.
\vspace{-.3cm}
\paragraph{Gaussian Families.}\mbox{} On a complete probability space $(\Om,\Fscr,\Pbb)$ and with respect to a filtration $\Fbb$ satisfying the usual conditions, we consider a family $\Xscr:=\{\N{\al},\ \al\in \Lm\}$ of continuous local martingales. We shall assume that $\Xscr$ is pairwise Gaussian, i.e., that the pair $(\X{\al},\X{\bt})$ of processes is Gausssian for every $\al, \bt\in\Lm$.
\bg{theorem}\label{Gaussian}
The family $\Xscr$ possesses the CRP relative to $\mb F^\Xscr$.
\e{theorem}
\bg{proof}
First we notice that $\Ebb[\X{\al}_{\,t}]$ is continuous and hence bounded in $t\in[0,T]$ (use characteristic functions to verify this).
Applying Fernique's theorem to the centred Gaussian random variable given by $(\X{\al}_t-\Ebb[\X{\al}_t])_{t\in[0,T]}$ with values in the space $C([0,T])$ of continuous real functions on $[0,T]$, we see that $\sup_{t\in[0,T]}|\X{\al}_t-\Ebb[\X{\al}_t]|$ and hence $\sup_{t\in[0,T]}|\X{\al}_t|$ is integrable. From this we can conclude that every continuous Gaussian local martingale $\X{\al}$ is actually a Gaussian martingale  and hence a process with independent increments\footnote{\,The authors are indebted to M.\ Urusov for pointing out the proof of this fact using Fernique's theorem}.
This yields that $\aP{\X{\al}}$ is deterministic. Applying this observation to $\X{\al}+\X{\bt}$ and $\X{\al}-\X{\bt}$ for $\al\not=\bt$, which are again Gaussian, by the polarization formula we obtain that $\aPP{\X{\al}}{\X{\bt}}$ is deterministic. Obviously, $\X{\al}_t$ has finite exponential moments of arbitrary order for every $t\geq0$ and $\al\in\Lm$. Finally, the compensated-covariation process $[\X{\al},\X{\bt}]-\aPP{\X{\al}}{\X{\bt}}$ equals $0$ because the martingales $\X{\al}$ and $\X{\bt}$ are continuous (cf. \eqref{def:cov}) for every $\al, \bt\in\Lm$. Thus $\Xscr\cup\{0\}$ is compensated-covariation stable. Now, upon using Theorem \ref{thm:den.pol}, Theorem \ref{thm:crp.ccs.fam} yields that $\Xscr\cup\{0\}$, and therefore also $\Xscr$, possesses the CRP.
\e{proof}
If moreover $\Xscr$ is a countable family of orthogonal martingales, say $\Xscr:=\{\X{n},\ n\geq1\}$, then it possesses the CRP on $L\p{\,2}(\Om,\Fscr\p\Xscr_T,\Pbb)$ and \eqref{eq:crp.or.sys} holds.

As a special case we get that a Brownian motion $W$ possesses the CRP relative to its natural filtration $\mb F^W$. This is the well-known result of \citet{I51} about the CRP of the Wiener process.

\vspace{-.3cm}

\paragraph{Poisson Families.}
Let $\Ns:=N-a(\cdot)$ be a compensated Poisson process with \emph{continuous} intensity function $a(\cdot)$ (cf.\ \citet{JS00}, Definition I.3.26). Clearly,  $\Ns$ is a square integrable martingale and we can calculate
\[
\textstyle[\overline N,\overline N]-\pb{\overline N}{\overline N}=\sum_{0\leq s\leq \cdot}(\Delta N_{\,s})\p{\,2}-a(\cdot)=\sum_{0\leq s\leq \cdot}\Delta N_{\,s}-a(\cdot)=\overline N\,.
\]
This observation leads to the following
\bg{theorem}
Let $\Xscr:=\{\N{\al},\ \al\in\Lm\}$ be a family of compensated Poisson processes on the probability space $(\Om,\Fscr\p\Xscr_T,\Pbb)$ and with respect to the filtration $\Fbb\p\Xscr$. If the family $\Xscr$ is pairwise independent then $\Xscr$ possesses the CRP relative to $\Fbb\p\Xscr$.
\e{theorem}
\bg{proof}
If $\al\not=\bt$, then $\X{\al}$ is independent of $\X{\bt}$ which implies that $[\X{\al},\X{\bt}]=0$ and hence also $\aPP{\X{\al}}{\X{\bt}}=0$.
As shown above, $[\X{\al},\X{\al}]-\aPP{\X{\al}}{\X{\al}}=\X{\al}$ and hence $\Xscr\cup\{0\}$ is compensated-covariation stable. The predictable covariations $\aPP{\X{\al}}{\X{\bt}}$ are deterministic for all $\al$ and $\bt$ from $\Lm$. The exponential moments of $\X{\al}_t$ are finite for every $t\geq 0$ and  $\al\in\Lm$. Upon using Theorem \ref{thm:den.pol}, we see that Theorem \ref{thm:crp.ccs.fam} can be applied and we conclude that $\Xscr$ has the CRP on $L\p2(\Om,\Fscr\p\Xscr_T,\Pbb)$.
\e{proof}
If $\Xscr$ is, moreover, a countable family, say $\Xscr:=\{\X{n},\ n\geq1\}$, then \eqref{eq:crp.or.sys} holds.

\begin{remark} We notice that the case of Gaussian and Poisson families can be unified in the following way. Let $\Xscr=\Yscr\cup\Zscr$ a family of square integrable martingales on the probability space $(\Om,\Fscr\p\Xscr_T,\Pbb)$ and with respect to the filtration $\Fbb\p\Xscr$. Suppose that $\Yscr$ is a pairwise Gaussian family of continuous martingales and $\Zscr$ a pairwise independent family of compensated Poisson processes. Then $\Xscr$ possesses the CRP. For the proof, upon recalling the arguments from above, we have only to remark that $[Y,Z]=0$ and hence $\aPP{Y}{Z}=0$ for all $Y\in\Yscr$ and $Z\in\Zscr$.
\end{remark}

\vspace{-.5cm}

\paragraph{Teugels Martingales.}\label{sec:T.ma} \
This example shows that under certain restrictions on the L\'evy measure $\nu$, Teugels martingales can be introduced as a family possessing the CRP on $L\p2(\Om,\Fscr_T\p L,\Pbb)$. Teugels martingales were considered in \citet{NS00} where it was proven that the orthogonalized Teugels martingales possess the CRP. We are going to obtain the CRP for the orthogonalized Teugels martingales as an application of the results of the present paper.

We fix a time horizon $[0,T]$, $T>0$.
Let $(L,\Fbb\p L)$ be a L\'evy process with characteristic triplet $(\bt,\sig\p{\,2},\nu)$. We require that there exist two constants $\ep,\lm>0$ such that the function
$x\mapsto\rme\p{\frac\lm2\,|x|}1_{\{|x|>\ep\}}$ is in $L\p{\,2}(\nu)$.
From this assumption it follows that the function
$x\mapsto x\p{\, n}$ belongs to $L\p1(\nu)\cap L\p{\,2}(\nu)$ for every $n\geq2$ and the identity $x\mapsto x$ is in $L\p{\,2}(\nu)$.
Moreover, the system $\{x\mapsto x\p{\, n},\ n\geq m\}$ is total in $L\p{\,2}(\nu)$ for every $m\geq1$. We set $h_1(x)=1_{\{0\}}+x$ and, for $n\geq2$, $h_n(x)=x\p{\,n}$, $x\in\Rbb$.
Because of $h_n\in L\p{\,2}(\mu)$ with $\mu:=\sig\p2\delta_0+\nu$, $n\geq1$, we can introduce the martingales $\X{h_n}=(\X{h_n}_{\,t})_{t\in[0,T]}$ as in \eqref{eq:en.mar}. With
$
\Tscr:=\{h_n,\ n\geq1\}\subseteq L\p{\,2}(\mu)
$,
we associate
$\Xscr_\Tscr:=\{\X{h_n},\ n\geq1\}$ which turns out to be the family of \emph{Teugels martingales}. From Theorem \ref{thm:tot.sys.h2.prp}, $\Xscr_\Tscr$ has the CRP on $L\p2(\Om,\Fscr_T\p L,\Pbb)$. We stress that, because of the assumptions on the L\'evy measure, the family of Teugels martingales is also compensated-covariation stable and satisfies the assumptions of Theorem \ref{thm:crp.ccs.fam}.
Let $\Pscr$ be the system of polynomials obtained by the Gram--Schmidt orthogonalization of $\Tscr$ in $L\p2(\mu)$. The associated family $\Xscr_\Pscr$ of martingales is the one of the \emph{orthogonalized} Teugels martingales: It possesses the CRP on $L\p2(\Om,\Fscr_T\p L,\Pbb)$ and the decomposition \eqref{eq:crp.or.sys.lp} of $L\p2(\Om,\Fscr_T\p L,\Pbb)$ holds.

\vspace{-.33cm}
\paragraph*{Hermite Polynomials.}\label{subsec:Her.pol} \
We consider a L\'evy process $(L,\Fbb\p L)$ with characteristic triplet $(\bt,\sig\p{\,2},\nu)$, where $\nu$ is of the form
\begin{equation}\label{def:eq.levy.meas}
\rmd\nu(x)=h(x)\,\rmd x,\qquad h(x)>0, \ x\in\Rbb\,.
\end{equation}
An important class of L\'evy processes with L\'evy measure as in \eqref{def:eq.levy.meas} is, for example, the class of $\al$-stable processes (see, e.g., \citet{S99}, Chapter 3).
We begin with the case of the Cauchy process.
A Cauchy process $(L,\Fbb\p L)$ is a purely non-Gaussian L\'evy process with characteristic triplet $(0,0,\nu)$ and $\rmd\nu(x):=x\p{-2}\,\rmd x$.
For a Cauchy process no moment exists and therefore Teugels martingales cannot be introduced.
We choose the function $g$ as $g(x):=x\exp(-x\p{\,2}/2)$, $x\in\Rbb$, and define $\rmd\nu\p g(x):=g\p{\,2}(x)\,\rmd\nu(x)=\exp(-x\p{\,2})\,\rmd x$ which is a finite measure.
Let $(H_n)_{n=0,1,\ldots}$ be the sequence of normalized Hermite polynomials.
The sequence $(H_n)_{n\geq0}$ is an orthonormal basis for $L\p{\,2}(\nu\p g)$. Therefore $\Tscr=\{C_n,\ n\geq0\}\subseteq L\p{\,2}(\nu)$, where
$C_n:=gH_n$, $n\geq0$, is a complete orthonormal system in $L\p{\,2}(\nu)$. Moreover, each $C_n$ is a bounded function. In view of Theorem \ref{thm:en.mar} and Theorem \ref{thm:cou.ort.sys.prp}, the family $\Xscr_\Tscr:=\{\X{C_n},\ n\geq0\}$ is a family of locally bounded orthogonal martingales with the CRP on $L\p2(\Om,\Fscr_T\p L,\Pbb)$ for which the decomposition \eqref{eq:crp.or.sys.lp} of $L\p2(\Om,\Fscr_T\p L,\Pbb)$ holds. Note that the system $\Tscr$ is not stable under multiplication, because the system of Hermite polynomials is not. Therefore, the family $\Xscr_\Tscr$ is not compensated-covariation stable.
For the general case of a L\'evy process with characteristic triplet  $(\bt,\sig\p{\,2},\nu)$ where $\nu$ is as in \eqref{def:eq.levy.meas}, we define $g\in L\p{\,2}(\nu)$ by $g(x):=(h(x)\p{-1/2})\,\exp(-x\p{\,2}/2)$, $x\in\Rbb$, and introduce the system $\Tscr:=\{P_n,\ n\geq0\}$, where $P_n:=g(0)H_n(0)1_{\{0\}}+1_{\Rbb\setminus\{0\}}gH_n$, $n\geq 0$, $(H_n)_{n\geq0}$ being the system of Hermite polynomials. The system $\Tscr$ is an orthogonal basis for $L\p2(\mu)$ and therefore the associated family $\Xscr_\Tscr:=\{\X{P_n},\ n\geq0\}$ is an orthogonal family of martingales possessing the CRP on $L\p2(\Om,\Fscr_T\p L,\Pbb)$ for which the decomposition \eqref{eq:crp.or.sys.lp} of $L\p2(\Om,\Fscr_T\p L,\Pbb)$ holds. In the general case, we cannot expect that $\Xscr_\Tscr$ consists of locally bounded martingales. Furthermore, $\Xscr_\Tscr$ is not compensated-covariation stable. Note that this variety of  examples is beyond earlier known results and techniques on the CRP.

\paragraph*{Haar Wavelet.}\label{subsec:Haar.w}
We consider a L\'evy process $(L,\Fbb\p L)$ with characteristics $(\bt,\sig\p2,\nu)$ and L\'evy measure as in \eqref{def:eq.levy.meas}.
Let $\lm$ be the Lebesgue measure on $(\Rbb,\Bscr(\Rbb))$ and $\psi\in L\p2(\lm)$. The function $\psi$ is called a \emph{wavelet} if the system of functions $\{\psi_{jk}: \psi_{jk}(x)=2\p {\frac j2}\,\psi(2\p j\,x-k),\ x\in\Rbb,\ j,k\in\Zbb\}$ is a complete orthonormal system of $L\p2(\lm)$. An example of a wavelet, known as \emph{Haar wavelet}, is the function defined by
$\psi(x):=1$, if $x\in[0,1/2)$; $\psi(x):=-1$ if $x\in[1/2,1)$ and $\psi(x):=0$ otherwise.
The system $\{\psi_{jk},\ j,k\in\Zbb\}$ generated by $\psi$ is the \emph{Haar basis} (cf.\ \citet{W97}, 1).
The system given by $\Tscr:=\{h\p{-1/2}(0)\psi_{jk}(0)1_{\{0\}}+1_{\Rbb\setminus\{0\}}h\p{-1/2}\psi_{jk},\ j,k\in\Zbb\}$ is a complete orthogonal system in $L\p2(\mu)$. The associated family $\Xscr_\Tscr:=\{\X{f},\ f\in\Tscr\}$ of orthogonal martingales possesses the CRP on $L\p2(\Om,\Fscr_T\p L,\Pbb)$ and the decomposition \eqref{eq:crp.or.sys.lp} of $L\p2(\Om,\Fscr_T\p L,\Pbb)$ holds. This method, leading to interesting families of martingales with the CRP not considered previously, can also be extended to L\'evy processes with arbitrary L\'evy measure $\nu$. 

\section*{Conclusions}We conclude with a short discussion concerning the relevance of the approach and the results given in the present paper. 

The first two examples about families of Gaussian continuous local martingales and pairwise independent Poisson processes are general formulations of existing results for Wiener and Poisson processes. Being immediate and straightforward applications of Theorem \ref{thm:crp.ccs.fam}, these examples demonstrate the power and flexibility of our approach to the chaotic representation property.
 
The most innovative examples of the present paper are those about the construction of families of martingales associated with a L\'evy process which do possess the CRP. Beside the classical case of Brownian motion and compensated Poisson process, only one example, the family of Teugels martingales, has been investigated before by Nualart and Schoutens in \cite{NS00}. However, this example requires the quite strong assumption on the underlying L\'evy process $L$ that $L_t$ ($t>0$) possesses a finite exponential moment of some order $\lm>0$. 

As we have seen in Section \ref{sec:appl}, the case of Teugels martingales can be directly deduced from the general approach given in Section 5 and applied to L\'evy processes in Section 6. It should be emphasized that the family of Teugels martingales (if it can be constructed) is only \textit{one example} of a great variety of families of martingales of a L\'evy process possessing the CRP. In addition, Theorem \ref{thm:tot.sys.h2.prp} (as a consequence of Theorem \ref{thm:crp.ccs.fam}) allows us to construct families of martingales possessing the CRP on $\Fbb\p L$ for an \textit{arbitrary} L\'evy process $L$, \textit{without any assumption} on the L\'evy measure $\nu$. This is outside the scope of other techniques and, in particular, those of Nualart and Schoutens \cite{NS00}. Applications are illustrated in the examples \emph{Hermite Polynomials} and \emph{Haar Wavelet}, where we investigated L\'evy processes for which the L\'evy measure $\nu$ only is equivalent to the Lebesgue me
 asure on $\Rbb$. This is a rather wide class of L\'evy processes, containing also processes, as the Cauchy process, which do not have any finite moments. For such cases it is impossible to follow the techniques of \citet{NS00}.

The assumption of the equivalence of $\nu$ to the Lebesgue measure is not important at all.  As proved in Theorem \ref{thm:tot.sys.h2.prp}, a necessary and sufficient condition for a family $\Xscr_\Tscr$ to possess the CRP with respect to $\Fbb\p L$ is that the system $\Tscr\subseteq L\p{\,2}(\mu)$ is total, $\mu:=\sig\p{\,2}\delta_{\,0}+\nu$, without further assumptions on $\nu$. In particular, such a family of martingales can be associated with any complete orthonormal system $\Tscr=\{f_n: \ n\geq 1\}$ of $L^2(\mu)$. This result is probably one of the most important of the present paper because it allows to construct numerous families of martingales possessing the CRP starting from the characteristics $(\bt,\sig\p{\,2},\nu)$ of a given L\'evy process $L$. We have only to look for adequate complete orthonormal systems $\Tscr$ in $L\p{\,2}(\mu)$. If moreover the L\'evy measure possesses further properties we can use them to construct special families of martingales with the CRP as we
  have seen above. 

This large degree of freedom in the construction of families with the CRP can play an important role in applications, e.g., for the problem of the completion of financial markets driven by  geometric L\'evy processes. In \citet{CNS05}, a L\'evy market driven by a geometric L\'evy process $L$ has been considered and, under the assumptions of \cite{NS00} on the L\'evy measure $\nu$, the market is completed involving \emph{compensated power-jump assets}, i.e., the Teugels martingales. In this context, the Teugels martingales have the interpretation of price processes associated with certain contingent claims that, if included in the market, make it complete, in the sense that any contingent claim of the market can be approximated (\emph{hedged}) trading in the stock and the Teugels martingales. However, by the results of the present paper it becomes clear that for \textit{any} geometric L\'evy market there are many alternative systems of martingales which can serve for its completion an
 d the question arises which are the ``most adequate'' systems from a theoretical and/or practical point of view.

In Section \ref{sec:def.it.int}, \ref{sec:ccs.f.mul.int} and \ref{sec:chaos} one of the most important assumptions on the family $\Xscr$ is that $\aPP{X}{Y}$ is deterministic, $X, Y\in\Xscr$. This hypothesis made possible the definition of iterated integrals ensuring natural isometry properties. On the other side, it is an important premise for the proof of Proposition \ref{prop:tot.term.rv.int} and hence Theorem \ref{thm:crp.ccs.fam}. A major extension of the approach would be to allow predictable covariations which are random. To explore adequate conditions should be the subject of future research.

Finally, we notice that the \emph{compensated-covariation stability} is not a necessary condition for the CRP as we have seen for L\'evy processes in Section 6 (see Theorem \ref{thm:tot.sys.h2.prp}). However, in the general part of the paper (Sections 4 and 5) this property plays an important role (see, e.g., Proposition \ref{MI} and Proposition \ref{prop:rep.pol}). To further exploit the relevance of the compensated-covariation stability for the CRP remains an open task for future research.
{
}
\end{document}